\documentclass[a4paper]{article}
\usepackage{import}
\usepackage{fullpage}

\usepackage{amsmath,amsthm, amssymb}
\usepackage{thmtools} %
\makeatletter
\@ifundefined{newcounteralias}{}{%
  \renewcommand\thmt@autorefsetup{\@xa\def\csname\thmt@envname autorefname\@xa\endcsname\@xa{\thmt@thmname}}%
}
\makeatother

\usepackage[hidelinks]{hyperref}
\usepackage[nameinlink]{cleveref}

\numberwithin{equation}{section}

\usepackage{bm}
\usepackage{xcolor}
\usepackage{mathtools}
\newcommand{\Rho}{P}

\newcommand{\seqi}{{\mathsf j}} %
\newcommand{\seqii}{{\mathsf k}} %

\newcommand{\R}{{\mathbb R}}
\newcommand{\Rd}{{\R^d}}

\DeclareMathOperator{\interior}{int}
\DeclareMathOperator*{\essinf}{ess\,inf}
\DeclareMathOperator*{\esssup}{ess\,sup}

\newcommand{\defeq}{\coloneqq}
\newcommand{\mob}{\mathrm{m}}
\newcommand{\mobup}{\mathrm{m}_{\uparrow}}
\newcommand{\mobdown}{\mathrm{m}_{\downarrow}}
\newcommand{\mobupwind}{\mathrm{m}_{\mathsf w}}

\DeclareMathOperator{\supp}{supp}
\DeclareMathOperator{\dist}{d}

\newtheorem{theorem}{Theorem}[section]
\newtheorem{lemma}[theorem]{Lemma}
\newtheorem{proposition}[theorem]{Proposition}
\newtheorem{corollary}[theorem]{Corollary}

\theoremstyle{definition}
\newtheorem{remark}[theorem]{Remark}

\usepackage{enumitem}

\newlist{enumeratesteps}{enumerate}{2}
\setlist[enumeratesteps,1]{
  label=\textit{Step \arabic*\protect\thisstep},
  ref=\arabic*,
  align=left, leftmargin=0pt, labelindent=!,
  listparindent=\parindent, labelwidth=0pt, itemindent=!
}
\setlist[enumeratesteps,2]{
  label= \textit{Step \theenumeratestepsi.\alph*\protect\thisstep},
  ref=\theenumeratestepsi.\alph*,
  align=left, leftmargin=0pt, labelindent=\parindent,
  listparindent=\parindent, labelwidth=0pt, itemindent=!
}
\newcommand{\step}[1][]{%
  \if\relax\detokenize{#1}\relax
  \def\thisstep{.}%
  \else
  \def\thisstep{: \protect #1.}%
  \fi
\item}

\Crefname{enumeratestepsi}{Step}{Steps}
\Crefname{enumeratestepsii}{Step}{Steps}

\newlist{enumeratecases}{enumerate}{2}
\setlist[enumeratecases,1]{
  label=\textit{Case \arabic*\protect\thiscase},
  ref=\arabic*,
  align=left, leftmargin=0pt, labelindent=!,
  listparindent=\parindent, labelwidth=0pt, itemindent=!
}
\setlist[enumeratecases,2]{
  label= \textit{Case \theenumeratecasesi.\alph*\protect\thiscase},
  ref=\theenumeratecasesi.\alph*,
  align=left, leftmargin=0pt, labelindent=\parindent,
  listparindent=\parindent, labelwidth=0pt, itemindent=!
}
\newcommand{\case}[1][]{%
  \if\relax\detokenize{#1}\relax
  \def\thiscase{.}%
  \else
  \def\thiscase{: \protect #1.}%
  \fi
\item}

\Crefname{enumeratecasesi}{Case}{Cases}
\Crefname{enumeratecasesii}{Case}{Cases}

\usepackage{float}
\usepackage{subcaption}
\DeclareCaptionLabelSeparator{nbspace}{~}
\usepackage[
  url=false,
  isbn=false,
  backend=biber,
  uniquename=init,
  giveninits,
  maxbibnames=100,
  doi=false,
  sortcites=true,
  maxalphanames=5,minalphanames=4,
  sorting=nyt,
  date = year,
  eprint=false,
  block=none]{biblatex}
\AtBeginRefsection{\GenRefcontextData{sorting=ynt}}
\AtEveryCite{\localrefcontext[sorting=ynt]}

\usepackage{csquotes}

\renewbibmacro{in:}{}
\DeclareFieldFormat[article]{citetitle}{#1}
\DeclareFieldFormat[article]{title}{#1}

\AtEveryBibitem{%
  \ifentrytype{online}
  {\clearfield{year}
    \clearfield{date}
    \clearfield{url}
    \clearfield{urldate}
    \clearfield{eprintclass}
    \clearfield{urlyear}
    \clearfield{pubstate}}
  {}%
}

\title{Convergence of a finite-volume scheme for \\ aggregation-diffusion equations \\ with saturation}
\author{David Gómez-Castro\thanks{Departamento de Matemáticas, Universidad Autónoma de Madrid, Spain \\ \& Instituto de Ciencias Matemáticas, Consejo Superior de Investigaciones Científicas, Spain}}

\begin{document}

\maketitle

\begin{abstract}
    In [Bailo, Carrillo, Hu. SIAM J. Appl. Math. 2023] the authors introduce a finite-volume method for aggregation-diffusion equations with non-linear mobility.
    In this paper we prove convergence of this method using an Aubin--Simon compactness theorem due to Gallouët and Latché. We use suitable $H^1$ and $W^{-1,1}$ discrete norms.

    We provide two convergence results.
    A first result shows convergence with general entropies ($U$) (including singular and degenerate) if the initial datum does not have free boundaries, the mobility is Lipschitz, and the confinement ($V$) and aggregation ($K$) potentials are $W^{2,\infty}_0$.
    A second result shows convergence when the initial datum has free boundaries, mobility is just continuous, and $V$ and $K$ are $W^{1,\infty}$, but under the assumption that the entropy $U$ is $C^1$ and strictly convex.

    \medskip


    Aggregation--Diffusion equations.
\end{abstract}

\setcounter{tocdepth}{1}
\tableofcontents

\section{Introduction and main results}
\begin{subequations}
  The aim of this paper is to prove convergence of a finite-volume numerical scheme for the family of problems
  \label{eq:main problem}
  \begin{align}
    \partial_t \rho & = \nabla \cdot ( \mob(\rho) \nabla \xi) \quad \text{ if } (t,x) \in (0,T) \times \Omega,
  \end{align}
  where $\xi$ is a non-local function of $\rho$,
  posed in a bounded open set $\Omega \subset \mathbb R^d$, and subject to no-flux conditions on the boundary, i.e.,
  \begin{equation}
    \mob(\rho) \nabla \xi \cdot n = 0 \quad \text{ if } (t,x) \in (0,T) \times \partial \Omega .
  \end{equation}
  We consider the case of aggregation-diffusion equations, i.e.,
  \begin{equation}
    \label{eq:xi}
    \xi (t,x) = U'(\rho(t,x)) + V(x) + \int_\Omega K(x,y) \rho(t,y) dy,
  \end{equation}
  where $U$ is convex, $V$ and $K$ are regular enough and $K$ is symmetric.
  Naturally, we impose an initial condition
  \begin{align}
    \rho(0,x) & = \rho_0(x) \in [0,\alpha] \quad \text{ if } x \in \Omega.
  \end{align}
\end{subequations}
We focus on the case of non-linear mobility of saturation type, i.e.,
\begin{equation}
  \text{there exists $\alpha > 0$ such that }
  \mob(0) = \mob(\alpha) = 0 \text{ and } \mob > 0 \text{ in } {(0,\alpha)}.
\end{equation}
We will discuss the case $\mob(s) = \mobup(s) \mobdown(s)$ where $\mobup$ is non-decreasing and $\mobdown$ is non-increasing, which is fairly general as given by \Cref{lem:mob decomposition} below.
Equation \eqref{eq:main problem} can be seen, when $K(x,y) = K(y,x)$, as formal Wasserstein-type gradient flow (see \cite{CarrilloLisiniSavareSlepcev2010}) of the free energy
\begin{equation}
  \label{eq:free energy}
  \mathcal F(\rho) = \int_\Omega U(\rho)  + \int_\Omega V \rho + \frac 1 2 \int_\Omega \int_\Omega K(x,y) \rho(x) \rho(y) dx dy.
\end{equation}
Notice that this energy is such that $\xi = \frac{\delta \mathcal F}{\delta \rho}$, i.e. the first variation of $\mathcal F$.

\bigskip

To be more precise, when $\mob$ is concave, this is a rigorous gradient flow for a generalized Wasserstein-type distance. However, if $\mob$ is not linear and $V$ or $K$ are non-trivial, the problem is not $\lambda$-displacement convex, so gradient-flow theory is not very useful.
When $\mob$ is not concave then the natural candidate for Wasserstein-type distance is not a distance (see \cite{CarrilloLisiniSavareSlepcev2010}). In either case, the free energy is dissipated along the evolution.

\bigskip

The family of equations \eqref{eq:main problem} has been significantly studied in different directions.
For a
survey in the case of linear mobility $\mob(\rho) = \rho$, see \cite{gomez-castroBeginnersGuideAggregationdiffusion2024} and the references therein.
Power-type nonlinear mobilities appear, for example, in the porous-medium equation with non-local pressure, see \cite{stan2019ExistenceWeakSolutions,stanFiniteInfiniteSpeed2016,CarrilloGomezCastroVazquez2022AIHP,CarrilloGomezCastroVazquez2022ANONA}.
For an extensive discussion on saturation-type problems we point the reader to the introduction of \cite{BailoCarrilloHu2023}.

\bigskip

In this article we discuss a scheme in general dimension $d\ge 1$ which is the natural extension of previous schemes in the literature:
for linear mobility $\mob(\rho) = \rho$ see \cite{CarrilloChertockHuang2015,BailoCarrilloHu2020},
for nonlinear mobility $\mob(\rho) = \rho \mobdown(\rho)$ see \cite{BailoCarrilloHu2023},
and for general $\mob(\rho) = \mobup(\rho) \mobdown(\rho)$ but $K = 0$ and $d=1$ see \cite{CarrilloFernandez-JimenezGomez-Castro}.
The explicit version of the numerical scheme discussed in this paper was shown to converge for the porous-medium equation with non-local pressure in dimension $d=1$ in \cite{deltesoConvergentFiniteDifferencequadrature2023}
(see also \cite{CarrilloGomezCastroVazquez2022AIHP,CarrilloGomezCastroVazquez2022ANONA}).
A finite-volume scheme for aggregation-diffusion over unstructured meshes for linear mobility $\mob(\rho) = \rho$, linear diffusion $U(\rho) = \rho \log \rho$, $V=0$, and $K(x,y) = W(x-y)$ can be found in \cite{SchlichtingSeis2022}.
A very interesting finite-volume scheme obtained by discretisation of the JKO scheme can be found in \cite{CancesGallouetTodeschi2020}, where convergence is rigorously proven for the Fokker-Planck case: linear mobility and diffusion, $V(x) = |x|^2/2$ and $K = 0$.
For a finite-difference scheme for the case of linear mobility see \cite{BurgerInzunzaMuletVillada2019}.
For particle-based methods in dimension $d=1$ see \cite{FrancescoRosini2015} and for $\mob(\rho) = \rho \mobdown (\rho)$ see \cite{daneriDeterministicParticleApproximation2020,DiFrancescoFagioliRadici2019,dimarinoOptimalTransportNonlinear2022,marconiStabilityQuasientropySolutions2022,radiciDeterministicParticleMethod2024}.

\medskip

The closest references for convergence results of this type of finite-volume schemes is \cite{BailoCarrilloMurakawaSchmidtchen2020,SchlichtingSeis2022,CancesGallouetTodeschi2020}.
Since we are working with non-linear diffusion $U$, our a priori estimates are closer to those in \cite{BailoCarrilloMurakawaSchmidtchen2020}.
In that paper the authors study the case of linear mobility $\mob(\rho) = \rho$, i.e., no saturation is present, the problem is in $d=1$ (and $d > 1$ by dimensional splitting), and the authors assume that $U \in C^2([0,\infty))$, $V, K$ in $C^2$, $U'' > 0$ in $(0,\infty)$, and $s^{-1} U'' \in L^1_{loc}([0,\infty))$. The compactness proofs in \cite{BailoCarrilloMurakawaSchmidtchen2020,SchlichtingSeis2022} are based on a Riesz--Fréchet--Kolmogorov argument. In this paper we will use an Aubin--Simon theorem introduced in \cite{Gallouet2012}.

\subsection{Numerical scheme}
For our scheme we assume that the mobility can be decomposed as
\begin{equation}
  \label{eq:mobility decomposition}
  \begin{aligned}
    \mob(\rho) & = \mobup (\rho) \mobdown(\rho) \quad \text{ for all } \rho \in [0,\alpha],
    \\
    &
    \mobup \text{ non-decreasing, }
    \mobdown \text{ non-increasing, and }
    \\
    & \mobup(0) = \mobdown(\alpha) = 0 \\
    & \mobup, \mobdown > 0 \text{ in } (0,\alpha)  .
  \end{aligned}
\end{equation}
This decomposition holds for a fairly general setting due to \Cref{lem:mob decomposition} (a minor improvement over the similar result in \cite{CarrilloFernandez-JimenezGomez-Castro}).
With this decomposition we define the up-winding of the mobility
\begin{equation}
  \mobupwind(a,b) = \mobup(a) \mobdown (b), \qquad \text{for } a,b \in [0,\alpha].
\end{equation}
This up-winding was already introduced in \cite{BailoCarrilloHu2023} in the case $\mobup(\rho) = \rho$. In \cite{CarrilloFernandez-JimenezGomez-Castro} this was extended for general $\mobup$ and the authors show that this choice makes the scheme monotone.
We fix a spatial step $\bm h = (h_1, \cdots, h_d)$  with $h_k > 0$, and we will use the spatial mesh for $\bm i \in \mathbb Z^d$
\begin{align*}
  \bm x^{(\bm h)}_{\bm i} \defeq (i_1 h_1, \cdots, i_d h_d), \qquad
  Q^{(\bm h)}_{\bm i}\defeq \prod_{k=1}^d \Big( (i_k - \tfrac {1} 2)h_k , (i_k + \tfrac {1} 2)h_k \Big), \qquad Q^{(\bm h)} \defeq Q^{(\bm h)}_{\bm 0}.
\end{align*}
We introduce the set of admissible indexes and the discrete domain
\begin{equation}
  \label{eq:definition of Omega_h}
  \bm I_{\bm h}(\Omega)
  \defeq \left\{ \bm i \in \mathbb Z^d : \overline{Q_{\bm i}^{(\bm h)}} \subset \Omega \right\}
  \qquad
  \text{ and }
  \qquad
  \Omega_{\bm h} \defeq \interior\left( \bigcup_{\bm i \in \bm I_{\bm h}} \overline{Q^{(\bm h)}_{\bm i}} \right).
\end{equation}
Following the idea in \cite{BailoCarrilloHu2023}, we consider the implicit finite-volume scheme on the grid $\bm I = \bm I_{\bm h} (\Omega)$,
using the convention that $\bm e_k \in \mathbb Z^d$ is the vector of entries $(\bm e_k)_i = \delta_{ik}$,
\begin{subequations}
  \label{eq:scheme}
  \begin{equation}
    \label{eq:scheme equation}
    \frac{\Rho_{\bm i}^{n+1} - \Rho_{\bm i}^n}{\tau}
    =
    - \sum_{k=1}^d \frac{F_{\bm i + \frac 1 2 \bm e_k}^{n+1}  - F_{\bm i - \frac 1 2 \bm e_k}^{n+1} }{h_k}, \qquad \text{for all } \bm i \in \bm I,
  \end{equation}
  where
  \begin{equation}
    \label{eq:scheme F,v,xi}
    \begin{aligned}
      F_{\bm i + \frac 1 2 \bm e_k}^{n+1}
      & \defeq
      \begin{dcases}
        \mobupwind(\Rho_{\bm i}^{n+1} , \Rho_{\bm i + e_k}^{n+1} ) (v_{\bm i + \frac 1 2 \bm e_k}^{n+1})_+                                                            \\
        \qquad + \mobupwind(\Rho_{\bm i + e_k}^{n+1}, \Rho_{\bm i}^{n+1}) (v_{\bm i + \frac 1 2 \bm e_k}^{n+1})_- & \text{if } {\bm i}, {\bm i + \bm e_k} \in \bm I , \\
        0                                                                                                         & \text{for all other } \bm i \in \mathbb Z^d,
      \end{dcases}
      \\
      v_{\bm i + \frac 1 2 \bm e_k}^{n+1}
      & \defeq
      \begin{dcases}
        -\frac{\xi_{\bm i + \bm e_k}^{n+1}-\xi_{\bm i}^{n+1}}{h_k} & \text{if } {\bm i}, {\bm i + \bm e_k} \in \bm I, \\
        0                                                          & \text{for all other } \bm i \in \mathbb Z^d,
      \end{dcases}
      \\
      \xi_{\bm i}^{n+1}
      & \defeq
      U'(\Rho_{\bm i}^{n+1}) + V_{\bm i} + |Q^{(\bm h)}| \sum_{\bm j \in \bm I} K_{\bm i, \bm j} \Rho_{\bm j}^{n + \frac 1 2}, \qquad \text{for all } \bm i \in \bm I,
    \end{aligned}
  \end{equation}
  where we use the convention $a_+ = \max\{a,0\}$ and $a_- = \min\{a,0\}$ so $a = a_+ + a_-$.
  For $\omega \subset \mathbb R^d$ measurable, we denote by $|\omega|$ its Lebesgue measure.
  In order to have decay of the energy for general $K$ we take
  \begin{equation}
    \label{eq:scheme P n + 1/2}
    \bm P^{n+\frac 1 2} \defeq \frac{\bm P^{n+1} + \bm P^n}{2}.
  \end{equation}
  To show convergence we select
  \begin{equation}
    \label{eq:scheme choice V and K}
    V_{\bm i} \defeq V(\bm x_{\bm i}^{(\bm h)}), \qquad K_{\bm i, \bm j} \defeq  K(\bm x_{\bm i}^{(\bm h)}, \bm x_{\bm j}^{(\bm h)}).
  \end{equation}
  As discussed in \cite{CarrilloFernandez-JimenezGomez-Castro}, this choice of up-winding makes the scheme monotone when $K = 0$.
  Lastly, we choose as initial value for the scheme the cell averages of the continuous initial condition $\rho_0$,
  \begin{equation}
    \Rho_{\bm i}^{0} = \frac{1}{|Q^{(\bm h)}|} \int_{Q^{(\bm h)}_{\bm i}} \rho_0(x) dx.
  \end{equation}
\end{subequations}

Occasionally, it will be useful to write $F$ as a multiple of $v$, and hence we introduce
\begin{equation}
  \label{eq:theta}
  \theta_{\bm i + \frac 1 2 \bm e_k}^{n+1} \defeq \mobupwind(\Rho_{\bm i}^{n+1} , \Rho_{\bm i + e_k}^{n+1} ) \mathbf{1}_{(0,+\infty)}(v_{\bm i + \frac 1 2 \bm e_k}^{n+1}) + \mobupwind(\Rho_{\bm i + e_k}^{n+1}, \Rho_{\bm i}^{n+1}) \mathbf{1}_{(-\infty,0)}(v_{\bm i + \frac 1 2 \bm e_k}^{n+1}),
\end{equation}
where $\mathbf{1}_A$ denotes the indicator function of the set $A$. With this choice we can write the decomposition
\begin{equation*}
  F_{\bm i + \frac 1 2 \bm e_k}^{n+1} = \theta_{\bm i + \frac 1 2 \bm e_k}^{n+1} v_{\bm i + \frac 1 2 \bm e_k}^{n+1}.
\end{equation*}

Since $\Omega$ is bounded and open, we have
\begin{equation}
  \tag{H0}
  \label{hyp:geometric assumption}
  \inf_{\omega \Subset \Omega} |\Omega \setminus \omega| = 0
\end{equation}
Given a solution to \eqref{eq:scheme} with $\tau = \tau_\seqi$, $\bm h = \bm h_\seqi$, and $\bm I = \bm I_{\bm h_\seqi} (\Omega)$ we define the piece-wise constant interpolation
\begin{equation}
  \label{eq:rho seqi}
  \rho_\seqi (t,x) =
  \begin{dcases}
    \Rho_{\bm i}^{n+1} & \text{if } x \in Q^{(\bm h_\seqi)} (\bm i) \text{ for } \bm i \in \bm I_{\bm h_\seqi} (\Omega) \text{ and } t \in [n\tau_\seqi, (n+1)\tau_\seqi), \\
    0                  & \text{otherwise}.
  \end{dcases}
\end{equation}
Throughout the manuscript, even though $\Rho_{\bm i}^{n+1}$ depends on $\bm h$ and $\tau$, when we deal with $\bm h_\seqi$ and $\tau_\seqi$, we will not include in the name of $\Rho_{\bm i}^{n+1}$ the dependence on $\seqi$ for the sake of clarity.

\subsection{Main results}
We prove two convergence results. The first result deals with problems where $\rho$ does not touch $0$ or $\alpha$.

The key idea is that if $\rho_0$ is away from $0$ and $\alpha$, the mobility is Lipschitz continuous, and $\nabla V$ and $\nabla_x K$ vanish on the boundary $\rho_t$ is never $0$ or $\alpha$.
Furthermore, we control the $\sup$ and $\inf$ of $\rho$ with exponential rates.
In that case, it suffices that the diffusion $U$ is non-degenerate in compact subsets of $(0,\alpha)$.
Some authors refer to the boundary of the regions where $\rho_t = 0$ or $\rho_t = \alpha$ as ``free boundary''. Our first result is one where we prove convergence and the absence of free boundaries.

\begin{theorem}
  \label{thm:convergence for data away from 0 mob Lipschitz}
  Let $T > 0$ be fixed, $\tau_\seqi \to 0$, $\bm h_\seqi \to 0$,
  $\bm \Rho$ be a solution to \eqref{eq:scheme}, and $\rho_\seqi$ be given by \eqref{eq:rho seqi}.
  Assume \eqref{hyp:geometric assumption}, that there exists $R>0$ with $\max_k(\bm h_\seqi)_k\le R\min_k(\bm h_\seqi)_k$ for all $\seqi$, and that
  \begin{enumerate}[label=\rm (H\arabic*),left=\parindent]
    \item \label{it:rho0 away from 0 and alpha}
      There exist constants $C_i$ such that $0 < C_1 \le \rho_0 \le C_2 < \alpha$.

    \item \label{it:mob decomposed with Lipschitz pieces}
      $\mob$ is of the form
      \eqref{eq:mobility decomposition} with $\mobup, \mobdown \in W^{1,\infty}(0,\alpha)$.

    \item \label{it:U C2 and convex}
      $U \in C^2((0,\alpha))$ with $U'' > 0$ in $(0,\alpha)$.

    \item \label{it:V and W compactly supported}
      $V \in W^{2,\infty}_{0} (\Omega)$, $K \in W^{2,\infty}_{0} (\Omega \times \Omega)$, and $K(x,y) = K(y,x)$.

  \end{enumerate}
  Then,
  \begin{enumerate}[label=\rm (C\arabic*), left=\parindent]
    \item \label{it:compactness}
      There exists a subsequence of $\rho_{\seqi}$ that converges in $L^1 ((0,T) \times \Omega)$ to a limit $\rho$.
    \item \label{it:limit is solution}
      This limit, $\rho$, is a weak solution to \eqref{eq:main problem}, i.e.,
      \begin{align*}
        \int_0^T \int_\Omega \left(-\rho\frac{\partial \varphi}{\partial t} + \mob(\rho) \nabla \xi \cdot \nabla \varphi  \right) = \int_\Omega \rho_0(x) \varphi(0,x) dx,
      \end{align*}
      for all $\varphi \in C^\infty([0,T) \times \overline \Omega)$ vanishing near $t=T$, where $\xi$ is given by \eqref{eq:xi}.
    \item \label{it:upper lower bound}
      Furthermore, $\rho$ does not have free boundaries and satisfies the bounds
      \[
        e^{- d \gamma t} \essinf_\Omega \rho_0 \le \rho_t
        \le
        (1 - e^{- d \gamma t}) \alpha + e^{- d \gamma t} \esssup_\Omega \rho_0
      \]
      where $\gamma = {C_R} \lambda \|\mob'\|_{L^\infty}$, $C_R=(3+R)/2$, and $\lambda = \|D^2 V\|_{L^\infty} + \|D_x^2 K\|_{L^\infty} \|\rho_0 \|_{L^1}$.
  \end{enumerate}
\end{theorem}
Recall that $W^{2,\infty}_0$ is the set of functions $f$ with bounded weak second derivative $D^2f$, such that $f$ and $\nabla f$ vanish on $\partial \Omega$.
In the previous theorem we must make very stringent assumptions on $\mob$, $V$, and $K$.
These come as a consequence of the generality of $U$, which might not be differentiable at $\rho = 0$ and $\rho = \alpha$ in unspecified ways (e.g., $\rho^m / (m-1)$ for $m \sim 0$).
Let us illustrate this fact by proving a result with more restrictive assumptions on $U$ but much weaker conditions on $\rho_0$, $\mob$, $V$, and $K$.
Given $U \in C^2((0,\alpha))$ we introduce the function $\Lambda_U$ by specifying
\begin{equation}
  \label{eq:Lambda_U}
  \Lambda_U''(s) = \frac{U''(s)}{\mob(s)} \text{ for all } s \in (0,\alpha), \qquad \Lambda_U(\tfrac \alpha 2) = \Lambda_U'(\tfrac \alpha 2) = 0.
\end{equation}
We know that $\Lambda \in C^2 ((0,\alpha))$, but we cannot make any claims of its regularity at $0$ or $\alpha$.
\begin{theorem}
  \label{thm:convergence with free boundary}
  Let $T > 0$ be fixed, $\tau_\seqi \to 0$, $\bm h_\seqi \to 0$,
  $\bm \Rho$ be a solution to \eqref{eq:scheme}, and $\rho_\seqi$ be given by \eqref{eq:rho seqi}.
  Assume that \eqref{hyp:geometric assumption} and
  \begin{enumerate}[label=\rm (H\arabic*'), left=\parindent]
    \item  \label{it:rho0 between 0 and alpha}
      $0 \le \rho_0 \le \alpha$.
    \item \label{it:mobility globally continuos and signs of m'}
      $\mob$ is of the form \eqref{eq:mobility decomposition} with $\mobup, \mobdown \in C ([0,\alpha])$.
    \item
      \label{it:U C1 [0,alpha] and U'' > C > 0}
      $U \in C^2((0,\alpha)) \cap C^1([0,\alpha])$ with $\inf_{s \in (0,\alpha)} U'' (s) > 0$, and $\Lambda_U \in C([0,\alpha])$.

    \item $V \in W^{1,\infty} (\Omega)$, $K \in W^{1,\infty} (\Omega \times \Omega)$, and $K(x,y) = K(y,x)$.

  \end{enumerate}
  Then, we have
  \ref{it:compactness} and \ref{it:limit is solution}.
\end{theorem}

In \cite{CarrilloFernandez-JimenezGomez-Castro}, the authors prove uniqueness of numerical solutions when $K = 0$ by a comparison principle. When $K \ne 0$, comparison does not hold in general%
, and we establish existence by a topological-degree argument.
\begin{proposition}
  \label{prop:existence}
  Under the assumptions of \Cref{thm:convergence for data away from 0 mob Lipschitz} or \Cref{thm:convergence with free boundary} the scheme admits a solution.
\end{proposition}
This proposition is proved in \Cref{sec:existence proof}.
Uniqueness of solutions is a more difficult problem, even at the discrete level.
For linear mobility $\mob(\rho) = \rho$ and $V, K \in W_c^{2,\infty}$ uniqueness of solutions of the continuous problem was proved in \cite{carrilloPartialMassConcentration2024} by elaborate fixed-point arguments. We expect a similar result will hold in the numerical setting, but it escapes the aims of this manuscript.

\begin{remark}[Examples]
  Consider $U(\rho) = \frac{\rho^m}{m-1}$ for $m \ne 1$ or $U(\rho) = \rho \log \rho$ for $m = 1$. Then \ref{it:U C2 and convex} holds for $m \in (0,\infty)$, whereas \ref{it:U C1 [0,alpha] and U'' > C > 0} holds for $m \in (1,2]$.
\end{remark}

\subsection{Intuition of the key estimates}

The proofs of \Cref{thm:convergence for data away from 0 mob Lipschitz,thm:convergence with free boundary} rely on a compactness argument using an Aubin-Simon result by Gallouët and Latché stated in \Cref{thm:gallouet}. This theorem is the discrete counterpart to the Aubin-Lions compactness theorem. In the context of our paper, this theorem states that if $\rho^{(n)} \in L^1((0,T) \times \Omega)$ is a sequence such that
\begin{equation}
  \label{eq:Aubin-Lions}
  \int_0^T \|\rho^{(n)}_t\|_{W^{1,p}(\Omega)}^q d t + \int_0^T \|\partial_t \rho^{(n)} _t\|_{W^{-1,1} (\Omega)}^q dt\le C
\end{equation}
with a uniform constant (where the negative Sobolev norm is given by \eqref{eq:negative Sobolev norm}),
then $\rho^{(n)}$ admits a subsequence that converges in $L^1( (0,T) \times \Omega)$.
To control the second term, we can use the dissipation of the free energy. At the continuous level, if
\[
  E[\rho] = \int_\Omega U(\rho) + \int_\Omega V \rho + \frac 1 2 \int_\Omega \int_\Omega K(x,y) \rho(x) \rho(y)
\]
then we have that, introducing the notation
\begin{equation}
  v_t(x) = -\nabla \xi(x,t)
\end{equation}
we get
\[
  \frac{d}{dt} E[\rho_t] = - \int_\Omega \mob(\rho) \left|v_t\right|^2.
\]
where $v_t(x) = {-}\nabla\left(U'(\rho(x)) + V(x) + \int_\Omega K(x,y) \rho(y)\,dy\right)$.
Therefore, we have that
\[
  \int_0^T \int_\Omega \mob(\rho) |v|^2 = E[\rho_0] - E[\rho_T].
\]
Therefore, at the continuous level we can use the definition of the $W^{-1,1}$ norm and Hölder inequality to deduce that
\begin{align*}
  \int_0^T \|\partial_t \rho\|_{W^{-1,1}}^2 &= \int_0^T \|\nabla \cdot (\mob(\rho) v) \|_{W^{-1,1}}^2 \\
  &\le \int_0^T \|\mob(\rho) v\|_{L^1}^2
  = \int_0^T \left(\int_\Omega \mob(\rho)^{\frac 1 2} \mob(\rho)^{\frac 1 2} |v|\right)^2 \\
  &\le |\Omega| \|\mob\|_{L^\infty(0,\alpha)} \int_0^T \int_\Omega m(\rho) |v|^2 \le C.
\end{align*}
This controls the second term in \eqref{eq:Aubin-Lions}.
To control the first term, we consider the evolution of an $L^p$-type energy
\begin{equation*}
  \frac{d}{dt} \int_\Omega \Lambda(\rho_t) = -\int_\Omega \Lambda'(\rho_t) \nabla \cdot (\mob(\rho_t) v_t) = \int_\Omega \mob(\rho_t) \Lambda''(\rho_t) \nabla \rho_t \cdot v_t = \int_\Omega \nabla H'(\rho_t) \cdot v_t.
\end{equation*}
where the last equality holds if $H''(\rho) = \mob(\rho) \Lambda''(\rho)$. Integrating in time we deduce that
\begin{equation}
  \label{eq:dissipation of Lp energies}
  \int_\Omega \Lambda (\rho_t)- \int_0^t \int_\Omega \nabla H'(\rho_s) \cdot v_s \, dx ds = \int_\Omega \Lambda (\rho_0).
\end{equation}
Given that $v_s$ includes a term $\nabla U'(\rho)$ we will be able to leverage this integral estimate into controls of $\nabla U'(\rho)$ in $L^1((0,T); L^q(\Omega))$ and, under some assumptions, of $\nabla \rho$.
In this equation we observe the terms
\[
  \mob(\rho_t) \Lambda''(\rho_t) U''(\rho_t) |\nabla \rho_t|^2.
\]
The key idea to control the first term in \eqref{eq:Aubin-Lions} is to select a good energy $\Lambda$.
When we can guarantee that free-boundaries do not form (i.e., $c \le \rho_t \le \alpha-c$ for some $c>0$) then this task is reasonably easy. When they form, we need to ask more of the diffusion.

\subsection{Structure of the paper}
In \Cref{sec:Aubin-Lions} we recall an Aubin--Simon compactness result for discrete settings introduced in \cite{Gallouet2012}, and we construct two discrete space norms that satisfy its hypothesis. We devote \Cref{sec:a priori estimates} to proving \emph{a priori} estimates in these norms that are suitable for the Aubin--Simon theorem.
In \Cref{sec:proof of convergence} we use these estimates to prove \Cref{thm:convergence for data away from 0 mob Lipschitz,thm:convergence with free boundary}.
We provide some numerical examples in \Cref{sec:numerical examples}.

\section{A discrete Aubin--Simon lemma and suitable norms}
\label{sec:Aubin-Lions}
We will apply the following Aubin--Simon compactness lemma due to Gallouët and Latché.

\begin{lemma}[{\protect\cite[Theorem 3.4]{Gallouet2012}}]
  \label{thm:gallouet}
  Let $(B, \|\cdot \|_B)$ be a Banach space,
  $B_\seqi$ a sequence of finite-dimensional subspaces of $B$,
  and $\| \cdot \|_{X_\seqi}, \| \cdot \|_{Y_\seqi}: B_\seqi \to \mathbb R$ be norms on $B_\seqi$ such that:
  \begin{enumerate}[label=\rm (GL\arabic*), left=\parindent]
    \item
      \label{it:gallouet H1}
      If we have a sequence $u_\seqi \in B_\seqi$ such that $\|u_\seqi\|_{X_\seqi}$ are bounded, then $u_\seqi$ has a subsequence that converges strongly in $B$.
    \item
      \label{it:gallouet H2}
      If $u_\seqi \in B_\seqi$ for each $\seqi$, $u_\seqi \to u$ strongly in $B$, and $\|u_\seqi\|_{Y_\seqi} \to 0$, then $u = 0$.
  \end{enumerate}
  Fix $T>0$ and integers $N_\seqi\to\infty$, and set $\tau_\seqi=T/N_\seqi$.
  Let $H_\seqi \subset L^\infty(0,T; B_\seqi)$ be the space of functions which are constant in time for $t \in (n \tau_\seqi, (n+1)\tau_\seqi)$, $0\le n<N_\seqi$.
  For $u \in H_\seqi$ we denote by $u^n$ its a.e. constant value on the $n$th interval.
  Assume that we have a sequence $u_\seqi \in H_\seqi$ such that, for some $q \in [1,\infty)$ and all $\seqi \in \mathbb N$
  \begin{equation*}
    \tau_\seqi \sum_{n=0}^{ {N_\seqi-1} } \| u_\seqi^n \|_{X_\seqi}^q + \tau_\seqi \sum_{n=0}^{ {N_\seqi-2} } \left\|  \frac{ u_\seqi^{n+1} - u_\seqi^n}{\tau_\seqi}  \right\|^q_{Y_\seqi} \le C.
  \end{equation*}
  Then $u_\seqi$ has a subsequence that converges strongly in $L^q(0,T; B)$.
\end{lemma}

In our setting, we take $\| \cdot \|_{X_\seqi}$ an approximation of the $L^\infty \cap H^1$ norm,
and $\|\cdot\|_{Y_\seqi}$ an approximation of the $W^{-1,1}$ norm.
We recall the definition of the norm of the negative Sobolev space $W^{-1,1} (\Omega)$:
\begin{equation}
  \label{eq:negative Sobolev norm}
  \| \rho \|_{W^{-1,1}(\Omega)} = \inf \left\{ \int_\Omega \sum_{k=1}^d |F_k| : \bm F = (F_k) \text{ with } F_k \in L^1 (\Omega) \text{ s.t. } \rho = \nabla \cdot \bm F \text{ in distributional sense} \right\}.
\end{equation}
For $\bm I \subset \mathbb Z^d$ finite let us introduce $\mathbb R^{\bm I}$ as the set of vector $(P_{\bm i})_{\bm i \in \bm I}$. It has dimension $\# \bm I$.
We recall the discrete version (see, e.g., \cite{CarrilloFernandez-JimenezGomez-Castro}) of the norm given by
\begin{lemma}
  Let $\bm I \subset \mathbb Z^d$ be finite. For $\bm P \in \mathbb R^{\bm I}$ let us define
  \begin{equation}
    \label{eq:W-1,1 defn}
    \| \bm \Rho \|_{W^{-1,1}_{\bm h}(\bm I)}
    \defeq
    \inf \left\{ |Q^{(\bm h)}| \sum_{k=1}^d \sum_{\bm i \in \mathbb Z^d} |F_{\bm i + \frac 1 2 \bm e_k}|
      :
    \bm F \text{ s.t. }\forall \bm i \in \bm I, \Rho_{\bm i} = \sum_{k=1}^d \frac{F_{\bm i + \frac 1 2 \bm e_k}-F_{\bm i - \frac 1 2 \bm e_k}}{h_k} \right\}.
  \end{equation}
  This is a well-defined norm on $\mathbb R^{\bm I}$.
\end{lemma}
\begin{proof}
  To see that this is a well-defined norm, let us show that we are taking $\inf$ over a non-empty set.
  It is clear that we can solve the system
  \begin{equation*}
    \begin{dcases}
      \sum_{k=1}^d \frac{2U_{\bm i} - U_{\bm i +\bm e_k} - U_{\bm i - \bm e_k}}{h_k^{2}} = \Rho_{\bm i} & \forall \bm i \in \bm I,                       \\
      U_{\bm i} = 0                                                                                 & \forall \bm i \in \mathbb Z^d \setminus \bm I.
    \end{dcases}
  \end{equation*}
  Then, we can let $F_{\bm i + \frac 1 2 \bm e_k} = - \frac{U_{\bm i + \bm e_k} - U_{\bm i}}{h_k}$.
  The triangular inequality is trivial. Lastly, if $\| \bm \Rho \|_{W^{-1,1}_{\bm h}(\bm I)}$ $= 0$ then there are admissible fluxes tending componentwise to zero, so the divergence constraint gives $\bm P=0$.
  This completes the proof.
\end{proof}
Similarly, we introduce the semi-norm for $\bm P \in \mathbb R^{\bm I}$
\begin{equation*}
  [\bm P ]_{H^1_{\bm h} (\bm I)}^2 \defeq |Q^{(\bm h)}| \sum_{k=1}^d
  \sum_{\substack{\bm i \in \bm I \text{ s.t.}  \\ \bm i + \bm e_k \in \bm I}}
  \left| \frac{\Rho_{\bm i + \bm e_k} - \Rho_{\bm i }} {h_k}\right|^2 .
\end{equation*}
Now we can apply \Cref{thm:gallouet} to our setting.
\begin{lemma}
  \label{lem:gallouet in our setting}
  Assume \eqref{hyp:geometric assumption} and let $B \defeq L^2(\Omega)$.
  Define the spaces of piece-wise constant functions over the mesh, extended by zero as in \eqref{eq:rho seqi},
  \begin{equation*}
    B_\seqi \defeq \left\{
      \rho \in L^\infty (\Omega)
      :
      {
        \begin{aligned}
          & \rho \text{ is constant a.e. on } Q^{(\bm h_\seqi)}_{\bm i}
          \text{ for all } \bm i \in \bm I_{\bm h_\seqi}(\Omega), \\
          & \rho = 0 \text{ a.e. on } \Omega \setminus \Omega_{\bm h_\seqi}
      \end{aligned}}
    \right\}.
  \end{equation*}
  For $\rho_\seqi \in B_\seqi$, let $\Rho^{(\seqi)}_{\bm i}$ be its a.e. constant value on $Q^{(\bm h_\seqi)}_{\bm i}$ and set $\bm \Rho^{(\seqi)} \defeq (\Rho^{(\seqi)}_{\bm i})_{\bm i \in \bm I_{\bm h_\seqi}(\Omega)}$.
  We use the representative taking these values at the cell centres and introduce the norms
  \begin{align}
    \label{eq:definition of X_j and Y_j}
    \| \rho_\seqi \|_{X_\seqi}
    & \defeq
    \| \rho_\seqi \|_{L^\infty}
    + [\bm \Rho^{(\seqi)}]_{H^1_{\bm h_\seqi} (\Omega)},
    \qquad \qquad
    \| \rho_\seqi \|_{Y_\seqi}
    \defeq \| \bm \Rho^{(\seqi)} \|_{W^{-1,1}_{\bm h_\seqi} (\Omega)}.
  \end{align}
  Here $\Omega$ abbreviates $\bm I_{\bm h_\seqi}(\Omega)$ in the discrete norms.
  The zero-extension condition identifies $B_\seqi$ with $\mathbb R^{\bm I_{\bm h_\seqi}(\Omega)}$.
  In particular, $\|\rho_\seqi\|_{Y_\seqi} = 0$ implies $\bm \Rho^{(\seqi)} = 0$ and hence $\rho_\seqi = 0$ a.e. on all of $\Omega$, so $\|\cdot\|_{Y_\seqi}$ is indeed a norm on $B_\seqi$.
  These norms satisfy the assumptions of \ref{it:gallouet H1} and \ref{it:gallouet H2}.
\end{lemma}
\begin{proof}
  Let us first prove \ref{it:gallouet H1}. Assume that $\rho_\seqi \in B_\seqi$ and $\sup_\seqi \|\rho_\seqi\|_{X_\seqi} < \infty$.
  Let $\omega \Subset \Omega$
  and $h_\seqi < \frac 1 {2\sqrt{d}} \dist(\omega, \partial \Omega)$.
  Then, for any simplex (i.e., a segment if $d=1$, a triangle if $d=2$, etc.) of adjacent vertices $x_{\bm i}^{(\bm h)}$ that intersects $\omega$, the complete simplex lies in $\Omega$.
  Let $u_\seqi$ be the piecewise linear interpolation over the simplices with vertex indices $\bm i_k=\bm i+\sum_{r=1}^k\bm e_{\pi(r)}$, $k=0,\ldots,d$, where $\pi$ ranges over coordinate permutations (nodal values outside $\bm I_{\bm h_\seqi}(\Omega)$ are zero).
  That is, define for $\theta_k \ge 0$ such that $\theta_0 + \cdots + \theta_d = 1$
  {
    \begin{align*}
      u_\seqi \left(\sum_{k=0}^d \theta_k \bm x_{\bm i_k}^{(\bm h_\seqi)} \right)
      & \defeq
      \sum_{k=0}^d \theta_k \Rho^{(\seqi)}_{\bm i_k}.
    \end{align*}
  }
  By construction, in this simplex ${\frac{\partial u_\seqi}{\partial x_{\pi(k)}} = \frac{\Rho^{(\seqi)}_{\bm i_k}-\Rho^{(\seqi)}_{\bm i_{k-1}}}{(\bm h_\seqi)_{\pi(k)}}}$.
  Due to the assumption \eqref{hyp:geometric assumption}, therefore, for $\seqi$ large enough, $\omega$ is contained in the union of the simplices and therefore
  \begin{equation*}
    \| u_{\seqi} \|_{L^\infty (\omega)} + \|\nabla u_\seqi\|_{L^2 (\omega)} \le { C_d \| \rho_\seqi \|_{X_\seqi}} \le C.
  \end{equation*}
  Here $C_d$ depends only on the dimension.
  This means that the sequence $u_{\seqi}$ is pre-compact in $L^2(\omega)$ for each $\omega \Subset \Omega$.
  Therefore, there exists a subsequence $u_{\seqi_\seqii}$ that converges in $L^2_{loc}(\Omega)$, and let $\rho$ be its limit.
  Since the $L^\infty (\omega)$ bound is independent of $\omega$, we have that $\rho \in  L^\infty (\Omega)$.
  For every $\omega \Subset \Omega$ we have that
  \begin{equation*}
    \| u_{\seqi_\seqii} - \rho\|_{L^2(\Omega)}
    \le
    \| u_{\seqi_\seqii} - \rho\|_{L^2(\omega)} + C |\Omega \setminus \omega|^{\frac 1 2} .
  \end{equation*}
  Therefore, we can estimate
  \begin{equation*}
    \limsup_{\seqii \to \infty } \| u_{\seqi_\seqii} - \rho\|_{L^2(\Omega)} \le C |\Omega \setminus \omega|^{\frac 1 2} .
  \end{equation*}
  Now we can take infimum in $\omega \Subset \Omega$ on the right-hand side to deduce that ${u_{\seqi_\seqii}} \to \rho$ in $L^2(\Omega)$.
  On the other hand, over each simplex we have that for a.e. $x$
  \begin{equation*}
    |u_\seqi (x) - \rho_\seqi(x)| \le {\sum_{k=1}^d \Big| \Rho^{(\seqi)}_{\bm i_k} - \Rho^{(\seqi)}_{\bm i_{k-1}} \Big|}.
  \end{equation*}
  In particular, for any compact subset and $\seqi$ large enough we have that
  \begin{equation*}
    \int_{\omega} |u_\seqi - \rho_\seqi|^2 \le C |{\bm h_{\seqi}}|.
  \end{equation*}
  Arguing as before we have that $\rho_{\seqi_\seqii} \to \rho$ in $L^2 (\Omega)$.
  This proves \ref{it:gallouet H1}.

  We now prove \ref{it:gallouet H2}. Assume that $\rho_\seqi \in B_\seqi$, $\rho_\seqi \to \rho$ in $L^2(\Omega)$ and $\|\rho_\seqi\|_{Y_\seqi} \to 0$.
  Since this norm is bounded, we have that there exists $\bm F^{(\seqi)}$ such that
  \begin{equation*}
    \Rho^{(\seqi)}_{\bm i} = \sum_{k=1}^d \frac{F_{\bm i + \frac 1 2 \bm e_k}^{(\seqi)} - F_{\bm i - \frac 1 2 \bm e_k}^{(\seqi)} }{(\bm h_\seqi)_k} \text{ for all } \bm i \in \bm I_{\bm h_\seqi} (\Omega) \text{ and } |Q^{(\bm h)}| \sum_{\bm i \in \mathbb Z^d} \sum_k |F_{\bm i + \frac 1 2 \bm e_k}^{(\seqi)} | \to 0.
  \end{equation*}
  Let $\varphi \in C^\infty_c (\Omega)$ and consider it extended by $0$ outside $\Omega$.
  For $\seqi$ large enough, every cell meeting $\supp\varphi$ is admissible. Then, we can estimate
  \begin{align*}
    |Q^{(\bm h_\seqi)}| \left| \sum_{\bm i \in \bm I_{\bm h_\seqi} (\Omega)} P^{(\seqi)}_{\bm i} \varphi(x_{\bm i}^{(\bm h_\seqi)}) \right|
    & =  |Q^{(\bm h_\seqi)}| \left| \sum_{\bm i \in \bm I_{\bm h_\seqi} (\Omega)} \sum_{k=1}^d \frac{F_{\bm i + \frac 1 2 \bm e_k}^{(\seqi)} - F_{\bm i - \frac 1 2 \bm e_k}^{(\seqi)} }{(\bm h_\seqi)_k} \varphi(x_{\bm i}^{(\bm h_\seqi)}) \right|
    \\
    & =  |Q^{(\bm h_\seqi)}| \left| \sum_{\bm i \in \mathbb Z^d} \sum_{k=1}^d \frac{F_{\bm i + \frac 1 2 \bm e_k}^{(\seqi)} - F_{\bm i - \frac 1 2 \bm e_k}^{(\seqi)} }{(\bm h_\seqi)_k} \varphi(x_{\bm i}^{(\bm h_\seqi)}) \right|
    \\
    & = |Q^{(\bm h_\seqi)}|\left| \sum_{\bm i \in \mathbb Z^d} \sum_{k=1}^d F^{(\seqi)}_{\bm i + \frac 1 2 \bm e_k} \frac{\varphi(x_{\bm i + \bm e_k}^{(\bm h_\seqi)} )- \varphi(x_{{\bm i}}^{(\bm h_\seqi)} )}{(\bm h_\seqi)_k} \right|
    \\
    & \le \| \nabla \varphi \|_{L^\infty(\Omega)} |Q^{(\bm h_\seqi)}| \sum_{\bm i \in \mathbb Z^d} \sum_{k=1}^d \left| F^{(\seqi)}_{\bm i + \frac 1 2 \bm e_k} \right|
    \\
    & \to 0.
  \end{align*}
  Since $\varphi$ is compactly supported, and $\rho_\seqi$ converges in $L^2 (\Omega)$ we pass to the limit above to show that
  $
  \int_\Omega \rho \varphi = 0.
  $
  Given that this holds for any $\varphi \in C^\infty_c (\Omega)$, we conclude that $\rho = 0$. And we have proven \ref{it:gallouet H2}.
\end{proof}

\section{A priori estimates}
\label{sec:a priori estimates}
We devote this section to proving bounds of the numerical solutions in the norms $X_\seqi$ and $Y_\seqi$ given by \eqref{eq:definition of X_j and Y_j}.
We first provide an extension of the decomposition result in \cite{CarrilloFernandez-JimenezGomez-Castro} that follows through the same proof.
\begin{lemma}
  \label{lem:mob decomposition}
  Assume that $\mob \in W^{1,1}_{loc} (0,\alpha)$,
  $\mob : (0,\alpha) \to (0,\infty)$,
  $\mob(0^+) = \mob(\alpha^-) = 0$,
  $\mob' \ge 0$ in $(0,\delta)$ and $\mob' \le 0$ in $(\alpha - \delta, \alpha)$ for some $\delta > 0$.
  Then, we have \eqref{eq:mobility decomposition} with $\mobup, \mobdown \in W^{1,1}_{loc}(0,\alpha)$.
  If $\mob \in W^{1,\infty} (0,\alpha)$, then $\mobup, \mobdown \in W^{1,\infty} (0,\alpha)$.
\end{lemma}
\begin{proof}
  Notice that we can make the choice
  \begin{equation*}
    \mobup (\rho) \defeq  \exp\left( \int_{\frac \alpha 2}^\rho \frac{(\mob'(s))_+}{\mob(s)} ds \right)
    \qquad \text{and} \qquad
    \mobdown (\rho) \defeq \mob(\tfrac \alpha 2)\exp\left( \int_{\frac \alpha 2}^\rho \frac{(\mob'(s))_-}{\mob(s)} ds \right)
  \end{equation*}
  Since $\mob$ is continuous in $(0,\alpha)$ and positive, it is strictly positive over compacts. So the $W^{1,1}_{loc}$ claims are trivial.
  Since $\mob' \ge 0$ in $(0,\delta)$, then $\mobdown$ is constant in $(0,\delta)$. Then $\mobup(0^+) = \mob(0^+) / \mobdown(\delta) = 0$. And similarly at $s = \alpha$.
  The $W^{1,\infty}$ is obvious from the representation.
\end{proof}

\subsection{Evolution of the \texorpdfstring{$\sup/\inf$}{sup/inf}}
The aim of this section is to study the evolution of the $\sup$ and $\inf$ of $\Rho$. At the continuous level we can expand
\[
  \partial_t \rho_t = -\mob(\rho_t) \nabla \cdot v_t - \mob'(\rho_t) \nabla \rho_t \cdot v_t.
\]
Therefore, we roughly expect that if $M_t \defeq \sup_x \rho_t (x)$ and this supremum is achieved at some points $\overline x_t$ then
\[
  \frac{\partial}{\partial t} M_t = -\mob(M_t) \nabla \cdot v(\overline x_t),
\]
and a similar result for the $\inf$.
This can be justified if the maximum is achieved at an interior point and there is very mild Sobolev regularity (see \cite{constantinGlobalSolutionsQuasilinear2002}).
We will use the up-winding to obtain a suitable analogue.

\begin{lemma}
  \label{lem:sup/inf}
  Let $\bm i_0 \in \bm I_{\bm h} (\Omega)$ be such that
  $
  \Rho^{n+1}_{\bm i_0} = \max_{\bm i } \Rho^{n+1}_{\bm i}
  $,
  then
  \begin{enumerate}
    \item If $\bm i_{0} + \bm e_k \in \bm I_{\bm h} (\Omega)$ then
      \begin{equation*}
        F_{\bm i_0 + \frac 1 2 \bm e_k}^{n+1}
        \ge
        \mob(\Rho^{n+1}_{\bm i_0} )
        v_{\bm i_0 + \frac 1 2 e_k}^{n+1}.
      \end{equation*}
    \item If $\bm i_{0} - \bm e_k \in \bm I_{\bm h} (\Omega)$ then
      \begin{equation*}
        F_{\bm i_0 - \frac 1 2 \bm e_k}^{n+1}
        \le
        \mob(\Rho^{n+1}_{\bm i_0} )
        v_{\bm i_0- \frac 1 2 e_k}^{n+1}.
      \end{equation*}
    \item Therefore, if $\bm i_{0} + \bm e_k , \bm i_{0} - \bm e_k \in \bm I_{\bm h} (\Omega)$ for every $k$, then
      \begin{equation}
        \label{eq:sup as limsup}
        \max_{\bm i } \Rho^{n+1}_{\bm i}
        \le \max_{\bm i } \Rho^{n}_{\bm i} - \tau \mob \left( \Rho_{\bm i_0}^{n+1} \right)  \sum_{k=1}^d \frac{v^{n+1}_{\bm i_0 + \frac 1 2 \bm e_k} - v^{n+1}_{\bm i_0 - \frac 1 2 \bm e_k}}{h_k}
      \end{equation}
  \end{enumerate}
  Analogous results hold for $\min$.
\end{lemma}

\begin{proof}
  \newcommand{\thesup}{{P_{\bm i_0}^{n+1}}}
  If ${\bm i,\bm i + \bm e_k \in \bm I_{\bm h}(\Omega)}$ then, noticing that $\mob(a) = \mobupwind(a,a)$ we can write
  \begin{align*}
    F_{\bm i + \frac 1 2 \bm e_k}^{n+1}
    & =
    \mob(\Rho_{\bm i}^{n+1}) v^{n+1}_{\bm i + \frac 1 2 \bm e_k}
    + \Big( \mobupwind (\Rho_{\bm i}^{n+1}, \Rho_{\bm i + \bm e_k}^{n+1})- \mobupwind(\Rho_{\bm i}^{n+1},\Rho_{\bm i}^{n+1}) \Big) (v^{n+1}_{\bm i + \frac 1 2 \bm e_k})_+
    \\
    & \quad + \Big( \mobupwind (\Rho_{\bm i + \bm e_k}^{n+1},\Rho_{\bm i}^{n+1}) - \mobupwind (\Rho_{\bm i}^{n+1}, \Rho_{\bm i}^{n+1}) \Big) (v^{n+1}_{\bm i + \frac 1 2 \bm e_k})_-.
  \end{align*}
  Recalling that $\mobupwind(a,b)$ is non-decreasing in $a$ and non-increasing in $b$, that $a_+ \ge 0$ and $a_- \le 0$ (we are using the convention that $a_-$ is the negative part with negative sign so that $a = a_+ + a_-$) we have that
  \begin{align*}
    F_{\bm i + \frac 1 2 \bm e_k}^{n+1}
    & \ge \mob(\Rho_{\bm i}^{n+1}) v^{n+1}_{\bm i + \frac 1 2 \bm e_k}
    + \Big( \mobupwind (\Rho_{\bm i}^{n+1}, \thesup)- \mobupwind(\Rho_{\bm i_0}^{n+1},\Rho_{\bm i}^{n+1}) \Big) (v^{n+1}_{\bm i + \frac 1 2 \bm e_k})_+
    \\
    & \quad + \Big( \mobupwind (\thesup,\Rho_{\bm i}^{n+1}) - \mobupwind (\Rho_{\bm i}^{n+1}, \Rho_{\bm i_0}^{n+1}) \Big) (v^{n+1}_{\bm i + \frac 1 2 \bm e_k})_-
  \end{align*}
  Evaluating at $i = i_0$ we prove the first claim.
  The second result can be proven analogously.
  Joining these two facts we get the last claim.
\end{proof}

We now use this estimate to construct a priori estimates when $\nabla V$ and $\nabla_x K$ vanish on the boundary of the domain.

\begin{lemma}
  \label{lem:free boundaries 1}
  Assume \ref{it:rho0 away from 0 and alpha}, \ref{it:U C2 and convex}, and \ref{it:V and W compactly supported}.
  Assume $\max_k h_k/\min_k h_k\le R$ and set $C_R=(3+R)/2$.
  Let
  \begin{equation*}
    \lambda = \| D^2 V \|_{L^\infty} + \|D_{x}^2 K\|_{L^\infty} \| \rho_0 \|_{L^1}.
  \end{equation*}
  Then, letting $\underline \Rho^n = \inf_{\bm i} \Rho_{\bm i}^n, \overline \Rho^n = \sup_{\bm i} \Rho_{\bm i}^n$ we have
  \begin{equation*}
    \underline \Rho^{n+1} + {C_R}\tau \lambda d \mob(\underline \Rho^{n+1}) \ge \underline \Rho^n \qquad
    \alpha - \overline \Rho^{n+1} + {C_R}\tau \lambda d \mob(\overline \Rho^{n+1}) \ge \alpha - \overline \Rho^n
  \end{equation*}
\end{lemma}

\begin{proof}
  Since $\bm I_{\bm h} (\Omega)$ is finite, there exists $\bm i_0 \in \bm I_{\bm h} (\Omega)$ such that
  \begin{equation*}
    \overline \Rho^{n+1} = \Rho_{\bm i_0}^{n+1}.
  \end{equation*}
  It suffices to check that
  \begin{equation}
    \label{eq:equivalent condition free boundaries 1}
    -\frac{F_{\bm i_0+\frac 1 2 \bm e_k}^{n+1} - F_{\bm i_0-\frac 1 2 \bm e_k}^{n+1}}{h_k} \le {C_R}\lambda \mob(\overline \Rho^{n+1}).
  \end{equation}
  Since we are at a point of local maximum and $U'$ is non-decreasing ($U$ is convex)
  \begin{align}
    \label{eq:diffusion at maximum point 1}
    \frac{U'(\Rho_{\bm i_0 + \bm e_k}^{n+1}) - U'(\Rho_{\bm i_0}^{n+1})}{h_k}  & \le 0 \text{ if } \bm i_0 + \bm e_k \in \bm I_{\bm h} (\Omega), \\
    \label{eq:diffusion at maximum point 2}
    \frac{U'(\Rho_{\bm i_0}^{n+1}) - U'(\Rho_{\bm i_0 - \bm e_k}^{n+1}) }{h_k} & \ge 0 \text{ if } \bm i_0 - \bm e_k \in \bm I_{\bm h} (\Omega)
  \end{align}
  We use \Cref{lem:sup/inf} and discuss cases.
  We now look at one direction, $k$, and observe there are four possibilities
  \begin{enumeratecases}
    \case [We have ${\bm i_0 - \bm e_k,\bm i_0 + \bm e_k \notin \bm I_{\bm h}(\Omega)}$]
    Then
    \begin{equation*}
      \frac{F_{\bm i_0 + \frac 1 2 e_k} - F_{\bm i_0 - \frac 1 2 e_k}}{h_k} = 0.
    \end{equation*}
    Then \eqref{eq:equivalent condition free boundaries 1} follows since $\lambda \ge 0$ and $\mob \ge 0$.

    \case [We have ${\bm i_0 + \bm e_k,\bm i_0 - \bm e_k \in \bm I_{\bm h}(\Omega)}$]
    Applying \Cref{lem:sup/inf} we have
    \begin{equation*}
      -\frac{F_{{\bm i_0}+\frac 1 2 \bm e_k}^{n+1} - F_{{\bm i_0}-\frac 1 2 \bm e_k}^{n+1}}{h_k}  \le {+} \mob(\overline \Rho^{n+1})\frac{\xi_{\bm i_0 + \bm e_k} + \xi_{\bm i_0 - \bm e_k} - 2 \xi_{\bm i_0}}{h_k^2}.
    \end{equation*}
    Due to \eqref{eq:diffusion at maximum point 1} and \eqref{eq:diffusion at maximum point 2}
    \begin{gather*}
      -\frac{2 U'(\Rho_{\bm i_0}^{n+1}) - U'(\Rho_{\bm i_0 + \bm e_k}^{n+1})- U'(\Rho_{\bm i_0 - \bm e_k}^{n+1}) }{h_k^{2}} \le 0 .
    \end{gather*}
    We also observe that, by Taylor expansion,
    \begin{align*}
      & \frac{V(x_{\bm i_0 + \bm e_k}) + V(x_{\bm i_0 - \bm e_k}) - 2 V(x_{\bm i_0})}{h_k^2}
      \\
      & =
      {\frac{1}{h_k^2}\int_{-h_k}^{h_k}(h_k-|s|)\frac{\partial^2 V}{\partial x_k^2}(x_{\bm i_0}+s\bm e_k)\,ds}
      \le \|D^2 V\|_{L^\infty}.
    \end{align*}
    A similar computation holds for $K$.
    Hence, we can estimate
    \begin{align*}
      & \frac{\xi_{\bm i_0 + \bm e_k} + \xi_{\bm i_0 - \bm e_k} - 2 \xi_{\bm i_0}}{h_k^2}
      \\
      & \le \frac{V(x_{\bm i_0 + \bm e_k}) + V(x_{\bm i_0 - \bm e_k}) - 2 V(x_{\bm i_0})}{h_k^2}
      \\
      & \quad + |Q^{(\bm h)}| \sum_{\bm j }\frac{K(x_{\bm i_0 + \bm e_k},x_{\bm j}) + K(x_{\bm i_0 - \bm e_k}, x_{\bm j}) - 2 K(x_{\bm i_0}, x_{\bm j})}{h_k^2} \frac{\Rho_{\bm j}^{n+1} + \Rho_{\bm j}^n}{2} \\
      & \le {\lambda}.
    \end{align*}

    \case[We have ${\bm i_0 - \bm e_k \in \bm I_{\bm h}(\Omega)}$ but ${\bm i_0 + \bm e_k \notin \bm I_{\bm h}(\Omega)}$]
    \label{case:previous inside next outside}
    Extend $V$ by zero and choose $z\in\overline{Q^{(\bm h)}_{\bm i_0+\bm e_k}}\setminus\Omega$. Since $\nabla V(z)=0$ and $\|x_{\bm i_0}-s h_k\bm e_k-z\|_\infty\le(1+s+R/2)h_k$ for $0\le s\le1$,
    {
    \begin{align*}
      \left|\frac{V(x_{\bm i_0 - \bm e_k}) - V(x_{\bm i_0})}{h_k^2}\right|
      & \le \frac1{h_k}\int_0^1|\partial_k V(x_{\bm i_0}-s h_k\bm e_k)-\partial_k V(z)|\,ds \\
      & \le \|D^2V\|_{L^\infty}\int_0^1(1+s+R/2)\,ds
      =C_R\|D^2V\|_{L^\infty}.
    \end{align*}
    Hessian norms use $\|A\|_{\infty\to\infty}=\max_k\sum_\ell|A_{k\ell}|$.
    }
    A similar computation holds for $K$.
    Then using \Cref{lem:sup/inf} and \eqref{eq:diffusion at maximum point 2} and the previous results
    \begin{align*}
      -\frac{F_{{\bm i_0}+\frac 1 2 \bm e_k}^{n+1} - F_{{\bm i_0}-\frac 1 2 \bm e_k}^{n+1}}{h_k}
      & = h_k^{-1} F_{{\bm i_0}-\frac 1 2 \bm e_k}^{n+1}
      \le  \mob(\overline \Rho^{n+1}) \frac{ \xi_{\bm i_0 - \bm e_k} - \xi_{\bm i_0}}{h_k^2}
      \\
      & \le \mob(\overline \Rho^{n+1}) \frac{V(x_{\bm i_0 - \bm e_k}) - V(x_{\bm i_0})}{h_k^2}
      \\
      & \quad + \mob(\overline \Rho^{n+1})  |Q^{(\bm h)}| \sum_{\bm j }\frac{K(x_{\bm i_0 - \bm e_k}, x_{\bm j}) - K(x_{\bm i_0}, x_{\bm j})}{h_k^2} \Rho_{\bm j}^{n+{\frac12}}
      \\
      & \le {C_R}\lambda \mob(\overline \Rho^{n+1}).
    \end{align*}

    \case[${\bm i_0 - \bm e_k \notin \bm I_{\bm h}(\Omega)}$ but ${\bm i_0 + \bm e_k \in \bm I_{\bm h}(\Omega)}$]
    Analogous to \Cref{case:previous inside next outside}. \qedhere
  \end{enumeratecases}

\end{proof}

\begin{lemma}
  \label{lem:free boundaries 2}
  In the hypothesis of \Cref{lem:free boundaries 1}, assume furthermore \ref{it:mob decomposed with Lipschitz pieces} and let $L \defeq \| \mob' \|_{L^\infty}$.
  Then
  \begin{equation*}
    \underline \Rho^{n} \ge (1 + {C_R}\lambda \tau d L )^{-n} \underline \Rho^0, \qquad \alpha - \overline \Rho^n \ge (1 + {C_R}\lambda \tau d L )^{-n} (\alpha - \overline \Rho^0).
  \end{equation*}
\end{lemma}
\begin{proof}
  Since $m(0) = 0$ we have that
  $
  \mob(s) \le L s.
  $
  Then, we have that
  $$
  \underline \Rho^{n+1} + {C_R} \tau \lambda L d \underline \Rho^{n+1} \ge \underline \Rho^n,
  $$
  We can now apply induction. We proceed similarly for $\alpha - \overline \Rho^n$.
\end{proof}

\begin{remark}
  Notice that \Cref{lem:free boundaries 1,lem:free boundaries 2} fail if we remove the condition that $\nabla V, \nabla_x K$ vanish on the boundary. For example, let us consider the example in dimension $d=1$ for $\Omega = (0,1)$ with $\mob(\rho) = \rho(1-\rho)$, $U(\rho) = \rho^2, V(x) = x, K(x,y) = 0$, and $\rho_0 = \frac 1 2$.
  Let us write $I_h = \{1,\cdots, N\}$.
  In this case $\lambda = 0$. Then we would follow that $\Rho_{i}^n = 0.5$ for all $i, n$. But we would then have $v_{i+\frac 1 2}^n = - 1$ and $F_{i + \frac 1 2}^n = - \mob(\frac 1 2)$ for ${i = 1, \cdots, N-1}$. And we have $F_{{\frac 1 2}}^n = F_{N+\frac 1 2}^n = 0$.
  Evaluating the equation at $i = 1$ we recover $0 = {\frac{1}{4h}}$.
\end{remark}

\subsection{Dissipation of a discrete version of the free energy \texorpdfstring{\eqref{eq:free energy}}{(\ref{eq:free energy})}}
In this section we discuss several alternatives to \eqref{eq:scheme P n + 1/2}:
\begin{enumerate}[label=\rm(O\arabic*), left=\parindent]
  \item
    \label{it:choice of rhon+1/2 for K pos def}
    $K_{\bm i \bm j}$ is positive semi-definite and $\Rho_{\bm i}^{n+\frac 1 2} = \Rho^{n+1}_{\bm i}$ for all $\bm i \in \bm I$,

  \item
    \label{it:choice of rhon+1/2 for K neg def}
    $K_{\bm i \bm j}$ is negative semi-definite and $\Rho_{\bm i}^{n+\frac 1 2} = \Rho_{\bm i}^n$ for all $\bm i \in \bm I$,

  \item
    \label{it:choice of rhon+1/2 for K generic}
    \eqref{eq:scheme P n + 1/2} and no assumption of $K_{\bm i \bm j}$.
\end{enumerate}
\begin{remark}
  As pointed out in \cite{BailoCarrilloHu2020}, the explicit schemes with $U'(\Rho^n)$ do not generally dissipate energy.
\end{remark}

We will specialise these problems differently depending on the nature of $V$ and $K$.
We consider the free energy
\begin{equation}
  \label{eq:free energy discrete}
  \begin{aligned}
    E_{\bm h}[\bm \Rho] & = |Q^{(\bm h)}|\sum_{\bm i} U(\Rho_{\bm i}) +  |Q^{(\bm h)}|\sum_{\bm i}  V_{\bm i} \Rho_{\bm i} + \frac 1 2 |Q^{(\bm h)}|^2 \sum_{\bm i, \bm j} \Rho_{\bm i} K_{\bm i\bm j} \Rho_{\bm j}.
  \end{aligned}
\end{equation}
Decay of this free energy was proven in \cite{BailoCarrilloHu2020} for linear mobility. Their proof works in our setting, since it does not rely on the structure of $\theta_{\bm i + \frac 1 2 \bm e_k}$. We reproduce it in our notation for the reader's convenience.
\begin{lemma}
  \label{lem:scheme free-energy dissipation}
  Assume that $\bm I \subset \mathbb Z^d$ is finite, $\bm P^n, (V_{\bm i})_{\bm i \in \bm I} \in \mathbb R^{\bm I}$, $(K_{\bm i, \bm j}) \in \mathbb R^{\bm I \times \bm I}$ is symmetric and that $U$ is convex.
  Let $\bm \Rho^{n+1} \in \mathbb R^{\bm I}$ be a solution to \eqref{eq:scheme equation}-\eqref{eq:scheme F,v,xi} and either \ref{it:choice of rhon+1/2 for K pos def}, \ref{it:choice of rhon+1/2 for K neg def}, or \ref{it:choice of rhon+1/2 for K generic}.
  Then, we have that
  \begin{equation}
    \label{eq:decay of free energy}
    \tau |Q^{(\bm h)}|\sum_{k=1}^d \sum_{\bm i \in \mathbb Z^d}  \theta_{\bm i + \frac 1 2 \bm e_k}^{n+1} |v_{\bm i + \frac 1 2 \bm e_k}^{n+1}|^2 \le E_{\bm h}[\bm \Rho^n] - E_{\bm h}[\bm \Rho^{n+1}],
  \end{equation}
  where $\theta^{n+1}_{\bm i + \frac 1 2 \bm e_k}$ is given by \eqref{eq:theta}.
  If $U = 0$ and \ref{it:choice of rhon+1/2 for K generic}, then equality holds in \eqref{eq:decay of free energy}.
\end{lemma}
\begin{proof}
  Multiplying in the equation by $\xi_{\bm i}^{n+1}$ and adding in $\bm i$ we get using the no-flux condition that
  \begin{equation}
    \label{eq:free energy integrated by parts}
    \begin{aligned}
      \sum_{\bm i} & \frac{\Rho_{\bm i}^{n+1} - \Rho_{\bm i}^{n}}{\tau} \Big( U'(\Rho_{\bm i}^{n+1}) + V_{\bm i} + |Q^{(\bm h)}| \sum_{\bm j}  K_{\bm i\bm j} \Rho_{{\bm j}}^{n + \frac 1 2} \Big) =\sum_{\bm i} \frac{\Rho_{\bm i}^{n+1} - \Rho_{\bm i}^{n}}{\tau} \xi_{\bm i}^{n+1} \\
      & = -\sum_{\bm i}\sum_{k=1}^d\frac{F^{n+1}_{\bm i + \frac 1 2 \bm e_k} - F_{\bm i -\frac 1 2 \bm e_k}^{n+1}}{h_k}  \xi_{\bm i}^{n+1}
      = \sum_{\bm i}\sum_{k=1}^d F_{\bm i + \frac 1 2 \bm e_k} \frac{ \xi_{\bm i + \bm e_k}^{n+1} -  \xi_{\bm i}^{n+1}}{h_k}                                                                                                                                                        \\
      & = - \sum_{\bm i}\sum_{k=1}^d \theta_{\bm i + \frac 1 2 \bm e_k}^{n+1} |v_{\bm i + \frac 1 2 \bm e_k}^{n+1}|^2.
    \end{aligned}
  \end{equation}
  Because $U$ is convex it holds that for any $a, b\ge 0$
  \begin{equation}
    \label{eq:convexity}
    U(b) \ge U(a) + U'(a) (b-a).
  \end{equation}
  Therefore we can write
  \begin{equation*}
    \sum_{\bm i} U(\Rho^{n+1}_{\bm i}) - \sum_{\bm i} U(\Rho^n_{\bm i}) \le \sum_{\bm i} U'(\Rho_{\bm i}^{n+1}) (\Rho_{\bm i}^{n+1} - \Rho_{\bm i}^{n}).
  \end{equation*}
  Clearly, due to linearity
  \begin{equation*}
    \sum_{\bm i} V_{\bm i} \Rho^{n+1}_{\bm i} - \sum_{\bm i} V_{\bm i }\Rho^n_{\bm i} = \sum_{\bm i} V_{\bm i} (\Rho_{\bm i}^{n+1} - \Rho_{\bm i}^{n})
  \end{equation*}
  For the last term, we distinguish three cases:
  \begin{enumerate}[left=\parindent]
    \item[\ref{it:choice of rhon+1/2 for K pos def}] If $K_{\bm i \bm j}$ is positive semi-definite, then
      \begin{align*}
        \sum_{\bm i} & {(\Rho_{\bm i}^{n+1} - \Rho_{\bm i}^{n})}
        \sum_{\bm j}  K_{\bm i\bm j} \Rho_{{\bm j}}^{n + 1}
        \\
        & =
        \sum_{\bm i \bm j} K_{\bm i \bm j} \Rho_{\bm i}^{n+1} \Rho_{\bm j}^{n+1} -\sum_{\bm i \bm j} K_{\bm i \bm j} \Rho_{\bm i}^n \Rho_{\bm j}^{n+1} \\
        & =
        \frac 1 2 \sum_{\bm i \bm j} K_{\bm i \bm j} \Rho_{\bm i}^{n+1} \Rho_{\bm j}^{n+1}
        - \frac 1 2  \sum_{\bm i \bm j} K_{\bm i \bm j} \Rho_{\bm i}^{n} \Rho_{\bm j}^{n}                                                              \\
        & \qquad
        + \frac 1 2 \left( \sum_{\bm i \bm j} K_{\bm i \bm j} \Rho_{\bm i}^{n+1} \Rho_{\bm j}^{n+1}
          - 2 \sum_{\bm i \bm j} K_{\bm i \bm j} \Rho_{\bm i}^{n+1} \Rho_{\bm j}^{n}
        + \sum_{\bm i \bm j} K_{\bm i \bm j} \Rho_{\bm i}^{n} \Rho_{\bm j}^{n} \right)
        \\
        & \ge  \frac 1 2 \sum_{\bm i \bm j} K_{\bm i \bm j} \Rho_{\bm i}^{n+1} \Rho_{\bm j}^{n+1}
        - \frac 1 2  \sum_{\bm i \bm j} K_{\bm i \bm j} \Rho_{\bm i}^{n} \Rho_{\bm j}^{n}          ,
      \end{align*}
      and we can use $\Rho^{n+\frac 1 2} = \Rho^{n+1}$.

    \item[\ref{it:choice of rhon+1/2 for K neg def}] Conversely, if $K_{\bm i \bm j}$ is negative semi-definite, we proceed analogously.

    \item[\ref{it:choice of rhon+1/2 for K generic}] If we make no assumption on $K$, we use symmetry
      \begin{align*}
        & \frac 1 2 \sum_{\bm i, \bm j}
        K_{\bm i \bm j}\Rho_{\bm i}^{n+1} \Rho_{\bm j}^{n+1} - \frac 1 2 \sum_{\bm i, \bm j} K_{\bm i \bm j}\Rho_{\bm i}^{n} \Rho_{\bm j}^{n}
        \\
        & = \frac 1 2 \sum_{\bm i, \bm j}  K_{\bm i \bm j}\Rho_{\bm i}^{n+1} \Rho_{\bm j}^{n+1}
        - \frac 1 2 \sum_{\bm i, \bm j} K_{\bm i \bm j}\Rho_{\bm i}^{n} \Rho_{\bm j}^{n}
        + \frac 1 2 \sum_{\bm i, \bm j} K_{\bm i \bm j}\Rho_{\bm i}^{n+1} \Rho_{\bm j}^{n}
        - \frac 1 2 \sum_{\bm i, \bm j} K_{\bm i \bm j}\Rho_{\bm i}^{n} \Rho_{\bm j}^{n+1}
        \\
        & = \sum_{\bm i, \bm j} K_{\bm i, \bm j}\left({\Rho_{\bm i}^{n+1} - \Rho_{\bm i}^n}\right)\frac{ \Rho_{\bm j}^{n+1} + \Rho_{\bm j}^n }{2}
        = \sum_{\bm i, \bm j} K_{\bm i, \bm j} \Rho_{{\bm j}}^{n + \frac 1 2}\left({\Rho_{\bm i}^{n+1} - \Rho_{\bm i}^n}\right).
      \end{align*}
  \end{enumerate}
  Therefore, we have shown that
  \begin{equation*}
    E_{\bm h}[\bm \Rho^n] - E_{\bm h}[\bm \Rho^{n+1}]
    \ge
    {-}\tau |Q^{(\bm h)}|
    \sum_{\bm i}  \frac{\Rho_{\bm i}^{n+1} - \Rho_{\bm i}^{n}}{\tau} \Big( U'(\Rho_{\bm i}^{n+1}) + V_{\bm i} + |Q^{(\bm h)}| \sum_{\bm j}  K_{\bm i\bm j} \Rho_{{\bm j}}^{n + \frac 1 2} \Big).
  \end{equation*}
  Combining this with \eqref{eq:free energy integrated by parts}, we conclude the result.
  Notice that if $U = 0$ and \ref{it:choice of rhon+1/2 for K generic} we only have equalities.
\end{proof}

\subsection{Time compactness}

\begin{lemma}
    \label{lem:time continuity}
    Let $\bm P^n, \bm P^{n+1}$ be solutions to \eqref{eq:scheme equation}--\eqref{eq:scheme P n + 1/2} and let $E_{\bm h}$ be given by \eqref{eq:free energy discrete}. Then, we have that
    \begin{equation}
        \label{eq:time continuity to W-1,1}
        \tau \left\| \frac{\bm \Rho^{n+1} - \bm \Rho^{n}}{\tau} \right\|_{W^{-1,1}_{\bm h} (\bm I)}^2 \le d{|Q^{(\bm h)}|} |\bm I| \|\mobup\|_{L^\infty}\|\mobdown\|_{L^\infty}
        \Big( E_{\bm h}[\bm \Rho^{n}] - E_{\bm h}[\bm \Rho^{n+1}] \Big).
    \end{equation}
\end{lemma}
\begin{proof}
    Using the free energy dissipation
    \begin{equation*}
        \begin{aligned}
            \left\| \frac{\bm \Rho^{n+1} - \bm \Rho^{n}}{\tau} \right\|_{W^{-1,1}_{\bm h} (\bm I)}
             & \le  |Q^{(\bm h)}| \sum_{k=1}^d \sum_{\bm i } |F^{n+1}_{\bm i + \frac 1 2 \bm e_k}|                                                            \\
             & = |Q^{(\bm h)}|  \sum_{k=1}^d \sum_{\bm i} \theta_{\bm i+\frac 1 2 \bm e_k}^{n+1}|v^{n+1}_{\bm i + \frac 1 2 \bm e_k}|                         \\
             & \le \left(|Q^{(\bm h)}| \sum_{k=1}^d\sum_{\bm i } \theta_{\bm i+\frac 1 2 \bm e_k}^{n+1} \right) ^{\frac 1 2}
            \left( |Q^{(\bm h)}|  \sum_{k=1}^d \sum_{\bm i} \theta_{\bm i+\frac 1 2 \bm e_k}^{n+1}|v^{n+1}_{\bm i + \frac 1 2 \bm e_k}|^2 \right)^{\frac 1 2} \\
             & \le \left(d|Q^{(\bm h)}| |\bm I| \|\mobup\|_{L^\infty}\|\mobdown\|_{L^\infty} \right)^{\frac 1 2}
            \left( \frac{E_{\bm h}[\bm \Rho^{n}] - E_{\bm h}[\bm \Rho^{n+1}]}{\tau} \right)^{\frac 1 2}
            .
        \end{aligned}
    \end{equation*}
    Taking the power $2$ we recover the estimate.
\end{proof}
\begin{remark}
    Notice that due to \eqref{hyp:geometric assumption} we have that $|Q^{(\bm h_\seqi)}| |\bm I_{\bm h_\seqi} (\Omega)| \to |\Omega|$ as $\bm h_\seqi \to 0$.
\end{remark}

\subsection{Dissipation formulas for other energies}
The aim of this section is to provide a discrete version of \eqref{eq:dissipation of Lp energies}.

\begin{lemma}
  \label{prop:Sobolev-type estimate}
  Assume that $\bm I \subset \mathbb Z^d$ is finite, $\bm P^n, (V_{\bm i})_{\bm i \in \bm I} \in \mathbb R^{\bm I}$, $(K_{\bm i, \bm j}) \in \mathbb R^{\bm I \times \bm I}$.
  Let $\bm \Rho^{n+1} \in \mathbb R^{\bm I}$ be a solution to \eqref{eq:scheme equation}-\eqref{eq:scheme F,v,xi}.
  Furthermore, assume that $\Rho_{\bm i}^{n}, \Rho_{\bm i}^{n+1} \in [\underline \Rho, \overline \Rho]$ for $\bm i \in \bm I$,
  $H{,\Lambda} \in C^{1}([\underline \Rho, \overline \Rho])$ are convex with absolutely continuous first derivatives, such that
  \[
    \mob(s) \Lambda''(s) = H''(s) \text{ for } {\text{a.e. }} s \in [\underline \Rho, \overline \Rho].
  \]
  Then
  \begin{equation}
    \label{eq:energy control with general H}
    |Q^{(\bm h)}| \sum_{\bm i \in \bm I} \Lambda(\Rho_{\bm i}^{n+1})
    - \tau |Q^{(\bm h)}| \sum_{k=1}^d \sum_{\substack{\bm i \in \bm I \text{ s.t.}  \\ \bm i + \bm e_k \in \bm I}}  \frac{ H'(\Rho_{\bm i + \bm e_k}^{n+1}) - H'(\Rho_{\bm i}^{n+1})}{h_k}v^{n+1}_{\bm i + \frac 1 2 \bm e_{k}}
    \le
    |Q^{(\bm h)}| \sum_{\bm i \in \bm I} \Lambda(\Rho_{\bm i}^{n}).
  \end{equation}
\end{lemma}

\begin{proof}
  We define the generalised entropic average
  \begin{equation}
    \label{eq:generalised entropic average}
    \mu(x,y) =
    \begin{dcases}
      \frac{H'(x) - H'(y)}{\Lambda'(x) - \Lambda'(y)} & {\Lambda'(x) \ne \Lambda'(y)}, \\
      {\mobupwind(\min\{x,y\},\max\{x,y\})} & {\Lambda'(x) = \Lambda'(y)}.
    \end{dcases}
  \end{equation}
  This function is symmetric, i.e., $\mu(x,y) = \mu(y,x)$.
  Furthermore, if $x \le y$ then, using the decomposition of $\mob$ and absolute continuity,
  {
  \begin{equation*}
    \mobupwind(x,y)[\Lambda'(y)-\Lambda'(x)]
    \le H'(y)-H'(x)=\int_x^y \mob(s)\Lambda''(s)\,ds
    \le \mobupwind(y,x)[\Lambda'(y)-\Lambda'(x)].
  \end{equation*}
  Thus $H'(y)-H'(x)=\mu(x,y)[\Lambda'(y)-\Lambda'(x)]$, also when both differences vanish.
  }
  Discussing similarly the different cases, one recovers for all ${x,y\in[\underline \Rho,\overline \Rho]}$ that
  \begin{equation}
    \label{eq:generalised entropy is between max and min}
    \mobupwind\Big(\min\{x,y\}, \max\{x,y\} \Big) \le \mu(x,y) \le \mobupwind\Big( \max\{x,y\}, \min\{x,y\} \Big)
  \end{equation}
  Since $\Lambda$ is convex we can estimate using \eqref{eq:convexity} and the no-flux condition
  \begin{align*}
    \frac{1}{\tau} \Bigg( \sum_{\bm i \in \bm I} \Lambda(\Rho_{\bm i}^{n+1})
    & -   \sum_{\bm i \in \bm I} \Lambda(\Rho_{\bm i}^{n})  \Bigg)
    \le   \sum_{\bm i \in \bm I} \Lambda'(\Rho_{\bm i}^{n+1}) \frac{\Rho^{n+1}_{\bm i} - \Rho^n_{\bm i}}{\tau}
    \\
    & = -  \sum_{\bm i \in \bm I} \sum_{k=1}^d \Lambda'(\Rho_{\bm i}^{n+1}) \frac {F^{n+1}_{\bm i + \frac 1 2 \bm e_{k}} - F^{n+1}_{\bm i - \frac 1 2 \bm e_k}}{h_k} \\
    & =
    \sum_{k=1}^d
    \sum_{\substack{\bm i \in \bm I \text{ s.t.}
    \\ \bm i + \bm e_k \in \bm I}}
    \frac{\Lambda'(\Rho_{\bm i + \bm e_k}^{n+1}) - \Lambda'(\Rho_{\bm i }^{n+1})}{h_k} F^{n+1}_{\bm i + \frac 1 2 \bm e_{k}}.
  \end{align*}
  We expand the last term using the up-winding and adding and subtracting terms
  \begin{align*}
    & \frac{\Lambda'(\Rho_{\bm i + \bm e_k}^{n+1}) - \Lambda'(\Rho_{\bm i}^{n+1})}{h_k} F^{n+1}_{\bm i + \frac 1 2 \bm e_{k}}
    \\
    & \quad =
    \frac{\Lambda'(\Rho_{\bm i + \bm e_k}^{n+1}) - \Lambda'(\Rho_{\bm i}^{n+1})}{h_k} \mu\Big(\Rho_{\bm i + \bm e_k}^{n+1},\Rho_{\bm i}^{n+1}\Big) v^{n+1}_{\bm i + \frac 1 2 \bm e_{k}}
    \\
    & \qquad + \frac{\Lambda'(\Rho_{\bm i + \bm e_k}^{n+1}) - \Lambda'(\Rho_{\bm i}^{n+1})}{h_k} \Bigg( \mobupwind(\Rho_{\bm i}^{n+1},\Rho_{\bm i + \bm e_k}^{n+1}) - \mu\Big(\Rho_{\bm i + \bm e_k}^{n+1},\Rho_{\bm i}^{n+1}\Big) \Bigg) \Big(v^{n+1}_{\bm i + \frac 1 2 \bm e_{k}}\Big)_+ \\
    & \qquad + \frac{\Lambda'(\Rho_{\bm i + \bm e_k}^{n+1}) - \Lambda'(\Rho_{\bm i}^{n+1})}{h_k} \Bigg( \mobupwind(\Rho_{\bm i+\bm e_k}^{n+1},\Rho_{\bm i}^{n+1}) - \mu\Big(\Rho_{\bm i + \bm e_k}^{n+1},\Rho_{\bm i}^{n+1}\Big) \Bigg) \Big(v^{n+1}_{\bm i + \frac 1 2 \bm e_{k}}\Big)_- .
  \end{align*}
  We again distinguish cases:
  \begin{enumeratecases}
    \case [$\Rho_{\bm i}^{n+1} \le \Rho_{\bm i + \bm e_k}^{n+1}$] Then, since $\Lambda'' \ge 0$ and \eqref{eq:generalised entropy is between max and min} we have that
    \begin{gather*}
      \frac{\Lambda'(\Rho_{\bm i + \bm e_k}^{n+1}) - \Lambda'(\Rho_{\bm i}^{n+1})}{h_k} \ge 0, \\
      \mobupwind\left(\Rho_{\bm i}^{n+1},\Rho_{\bm i + \bm e_k}^{n+1}\right) - \mu\left(\Rho_{\bm i + \bm e_k}^{n+1},\Rho_{\bm i}^{n+1}\right) \le 0
      \qquad  \mobupwind\left(\Rho_{\bm i+\bm e_k}^{n+1},\Rho_{\bm i}^{n+1}\right)  - \mu\left(\Rho_{\bm i + \bm e_k}^{n+1},\Rho_{\bm i}^{n+1}\right) \ge 0
    \end{gather*}
    Thus, using that $a_- \le 0 \le a_+$ we have that
    \begin{align*}
      \frac{\Lambda'(\Rho_{\bm i + \bm e_k}^{n+1}) - \Lambda'(\Rho_{\bm i}^{n+1})}{h_k} F^{n+1}_{\bm i + \frac 1 2 \bm e_{k}}
      & \le
      \frac{\Lambda'(\Rho_{\bm i + \bm e_k}^{n+1}) - \Lambda'(\Rho_{\bm i}^{n+1})}{h_k} \mu\left(\Rho_{\bm i + \bm e_k}^{n+1},\Rho_{\bm i}^{n+1}\right) v^{n+1}_{\bm i + \frac 1 2 \bm e_{k}} \\
      & =
      \frac{H'(\Rho_{\bm i + \bm e_k}^{n+1}) - H'(\Rho_{\bm i}^{n+1})}{h_k} v^{n+1}_{\bm i + \frac 1 2 \bm e_{k}}.
    \end{align*}

    \case[$\Rho_{\bm i + \bm e_k}^{n+1} {<} \Rho_{\bm i}^{n+1}$] Then, we have similarly that
    \begin{gather*}
      \frac{\Lambda'(\Rho_{\bm i + \bm e_k}^{n+1}) - \Lambda'(\Rho_{\bm i}^{n+1})}{h_k} \le 0, \\
      \mobupwind(\Rho_{\bm i}^{n+1},\Rho_{\bm i + \bm e_k}^{n+1}) - \mu\Big(\Rho_{\bm i + \bm e_k}^{n+1},\Rho_{\bm i}^{n+1}\Big) \ge 0,
      \qquad  \mobupwind(\Rho_{\bm i+\bm e_k}^{n+1},\Rho_{\bm i}^{n+1})  - \mu\Big(\Rho_{\bm i + \bm e_k}^{n+1},\Rho_{\bm i}^{n+1}\Big) \le 0,
    \end{gather*}
    and the same inequality holds.
  \end{enumeratecases}
  Eventually, we have shown that
  \begin{align*}
    \frac{1}{\tau} \Bigg( \sum_{\bm i \in \bm I} \Lambda(\Rho_{\bm i}^{n+1})
    & -  \sum_{\bm i \in \bm I} \Lambda(\Rho_{\bm i}^{{n}}) \Bigg)
    \le {\sum_{k=1}^d \sum_{\substack{\bm i \in \bm I \text{ s.t.} \\ \bm i+\bm e_k\in\bm I}}} \frac{H'(\Rho_{\bm i + \bm e_k}^{n+1}) - H'(\Rho_{\bm i}^{n+1})}{h_k} v^{n+1}_{\bm i + \frac 1 2 \bm e_{k}}.
  \end{align*}
  Multiplying by $\tau|Q^{(\bm h)}|$, the result follows.
\end{proof}

\begin{remark}
  The generalised entropic average \eqref{eq:generalised entropic average} is a curious construction. It is not directly related to the upwinding $\mobupwind$ as shown by the following example, which we construct for simplicity in the case of linear mobility, and hence without saturation.
  In this case let $\mobup(\rho) = \rho$ and $\mobdown(\rho) = 1$.
  Consider $\Lambda'(\rho) = \rho$.
  Then $H''(\rho) = \rho$ and hence $H'({\rho}) = \rho^2 /2$. In this example
  $\mobupwind (x,y) = x$, whereas
  \[
    \mu(x, y) =
    \begin{cases}
      \tfrac 1 2 (x+y) & \text{if } x \ne y, \\
      x & \text{if } x = y
    \end{cases}
  \]
\end{remark}

\subsection{Spatial compactness}

We can now prove the following result
\begin{corollary}
  \label{cor:U'(rho) is H1x}
  Under the assumptions of \Cref{thm:convergence for data away from 0 mob Lipschitz} or \Cref{thm:convergence with free boundary} we have that
  \begin{equation*}
    \begin{aligned}
      \tau \sum_{n=0}^{N-1}
      & |Q^{(\bm h)}|
      \sum_{k=1}^d
      \sum_{\substack{\bm i \in \bm I \text{ s.t.} \\ \bm i + \bm e_k \in \bm I}}
      \left| \frac{U'(\Rho_{\bm i + \bm e_k}^{n+1}) - U'(\Rho_{\bm i}^{n+1})}{h_k} \right|^2
      \\
      & \le
      2 |Q^{(\bm h)}|\sum_{\bm i \in \bm I} \left(\Lambda_{U}(\Rho_{\bm i}^{0}) - \Lambda_{U}(\Rho_{\bm i}^{N})\right)
      +
      {d|\Omega|} \Big( \|\nabla V\|_{L^\infty} + \|\nabla_{x} K\|_{L^\infty} \|\rho_0\|_{L^1} \Big)^{2} N \tau.
    \end{aligned}
  \end{equation*}
\end{corollary}

\begin{proof}
  Under \Cref{thm:convergence for data away from 0 mob Lipschitz}, apply \Cref{prop:Sobolev-type estimate} with $H=U$ on the density range given by \Cref{lem:free boundaries 2}.
  Under \Cref{thm:convergence with free boundary}, let us define $U_\varepsilon$ by the conditions
  \begin{equation*}
    U_\varepsilon'' = \frac{\mob}{\mob + \varepsilon} U'',
    \qquad U_\varepsilon(\tfrac{\alpha}2) = U(\tfrac \alpha 2),
    \qquad U_\varepsilon'(\tfrac{\alpha}2) = U'(\tfrac \alpha 2)
  \end{equation*}
  and, by analogy to \eqref{eq:Lambda_U}, set $\Lambda_\varepsilon=\Lambda_{U_\varepsilon}$ with
  \begin{equation*}
    \Lambda_\varepsilon''(\rho) = \frac{U_\varepsilon''(\rho)}{\mob(\rho)} = {\frac{U''(\rho)}{\mob(\rho) + \varepsilon}, \quad 0<\rho<\alpha, \qquad
    \Lambda_\varepsilon(\tfrac\alpha2)=\Lambda_\varepsilon'(\tfrac\alpha2)=0}.
  \end{equation*}
  Since $U$ is convex and $C^1$, $U''\in L^1(0,\alpha)$; the bounds $0\le U_\varepsilon''\le U''$ and $0\le\Lambda_\varepsilon''\le U''/\varepsilon$ allow us to apply \Cref{prop:Sobolev-type estimate} with $H=U_\varepsilon$.

  First, we apply \eqref{eq:energy control with general H} iteratively to recover that
  \begin{align*}
    & -\tau \sum_{n=0}^{N-1} |Q^{(\bm h)}|
    \sum_{k=1}^d
    \sum_{\substack{\bm i \in \bm I \text{ s.t.}
    \\ \bm i + \bm e_k \in \bm I}}
    \left( \frac{U_\varepsilon'(\Rho_{\bm i + \bm e_k}^{n+1}) - U_\varepsilon'(\Rho_{\bm i}^{n+1})}{h_k} \right) v_{\bm i + {\frac12}\bm e_k}^{n+1}
    \le
    |Q^{(\bm h)}| \sum_{\bm i \in \bm I} \left(\Lambda_{U_\varepsilon}(\Rho_{\bm i}^{0}) - \Lambda_{U_\varepsilon}(\Rho_{\bm i}^{N})\right)
  \end{align*}
  {
  Dominated convergence gives $U_\varepsilon'\to U'$ uniformly and, using $\Lambda_\varepsilon(r)=\int_{\alpha/2}^r(r-s)U''(s)/(\mob(s)+\varepsilon)\,ds$ and continuity of $\Lambda_U$, $\Lambda_\varepsilon(r)\to\Lambda_U(r)$ for $r\in[0,\alpha]$.
  Passing to the limit, we obtain in either case
  }
  \begin{align*}
    & -\tau \sum_{n=0}^{N-1} |Q^{(\bm h)}|
    \sum_{k=1}^d
    \sum_{\substack{\bm i \in \bm I \text{ s.t.}
    \\ \bm i + \bm e_k \in \bm I}}
    \left( \frac{U'(\Rho_{\bm i + \bm e_k}^{n+1}) - U'(\Rho_{\bm i}^{n+1})}{h_k} \right) v_{\bm i + {\frac12}\bm e_k}^{n+1}
    \le
    |Q^{(\bm h)}| \sum_{\bm i \in \bm I} \left(\Lambda_{U}(\Rho_{\bm i}^{0}) - \Lambda_{U}(\Rho_{\bm i}^{N})\right).
  \end{align*}
  Going back to the definition of the velocity \eqref{eq:scheme F,v,xi} we have that
  \begin{align*}
    &
    -\frac{U'(\Rho_{\bm i + \bm e_k}^{n+1}) - U'(\Rho_{\bm i}^{n+1})}{h_k}
    v_{\bm i + \tfrac 1 2 \bm e_k}^{n+1}
    \\
    & \qquad  = \left( \frac{U'(\Rho_{\bm i + \bm e_k}^{n+1}) - U'(\Rho_{\bm i}^{n+1})}{h_k}  \right)^2
    \\
    & \qquad \quad
    {+} \frac{U'(\Rho_{\bm i + \bm e_k}^{n+1}) - U'(\Rho_{\bm i}^{n+1})}{h_k}
    \left(
      \frac{V_{\bm i+\bm e_k} - V_{\bm i}}{h_k}
      + {|Q^{(\bm h)}|}\sum_{\bm j} \frac{K_{\bm i + \bm e_k, \bm j} - K_{\bm i, \bm j}}{h_k} \Rho_{\bm j}^{n+\tfrac 1 2}
    \right)                                                                                                          \\
    & \qquad  \ge \frac 1 2 \left( \frac{U'(\Rho_{\bm i + \bm e_k}^{n+1}) - U'(\Rho_{\bm i}^{n+1})}{h_k}  \right)^2
    \\
    & \qquad \quad
    - \frac 1 2
    \left|
    \frac{V_{\bm i+\bm e_k} - V_{\bm i}}{h_k}
    + {|Q^{(\bm h)}|}\sum_{\bm j} \frac{K_{\bm i + \bm e_k, \bm j} - K_{\bm i, \bm j}}{h_k} \Rho_{\bm j}^{n+\tfrac 1 2}
    \right|^{2}.
  \end{align*}
  By mass conservation, each drift term is bounded by $\|\nabla V\|_{L^\infty}+\|\nabla_x K\|_{L^\infty}\|\rho_0\|_{L^1}$. Summing over at most $d|\Omega|/|Q^{(\bm h)}|$ edges and $N$ time steps gives the result.
\end{proof}

\section{Proof of the convergence results}
\label{sec:proof of convergence}
\subsection{Proof of \texorpdfstring{\Cref{thm:convergence for data away from 0 mob Lipschitz}}{Theorem \ref{thm:convergence for data away from 0 mob Lipschitz}}}
In the proof we will avoid the index $\seqi$ in some elements, for the sake of readability.
We work up to $T_\seqi=\lceil T/\tau_\seqi\rceil\tau_\seqi$ and apply \Cref{thm:gallouet} by rescaling time.
\begin{enumeratesteps}
  \step[Upper and lower bounds.]
  \label{step:proof convergence data away from 0 - upper/lower bound}
  Let us denote $L = \|\mob'\|_{L^\infty}$.
  We consider
  \begin{equation}
    \label{eq:sub and supersolution discrete}
    \underline \Rho_\seqi^n =  (1 + {C_R}\tau_\seqi \lambda d L)^{-n} \essinf_\Omega \rho_0
    \qquad
    \overline \Rho_\seqi^n = \alpha - (1 + {C_R}\tau_\seqi \lambda d L)^{-n} \Big(\alpha - \esssup_\Omega \rho_0\Big)
  \end{equation}
  It is a simple computation that
  \begin{equation*}
    \underline \Rho_\seqi = \inf \{ \underline \Rho_\seqi^n : n \tau_\seqi \le {T_\seqi} \} > 0.
  \end{equation*}
  We similarly construct $\overline \Rho_\seqi < \alpha$.
  Due to \Cref{lem:free boundaries 2}, we prove recursively that $\underline \Rho_\seqi \le \rho_\seqi \le \overline \Rho_\seqi$ on $(0,T_\seqi)\times\Omega_{\bm h_\seqi}$.
  Since each element is positive, and the limit is positive $\underline \Rho = \inf_\seqi \underline \Rho_\seqi > 0$. Similarly, we build $\overline \Rho$.

  \step[Bound in the $X_\seqi$-norm]
  \label{step:proof convergence data away from 0 - bound Xj}
  We define $\mu = \min_{[\underline \Rho, \overline \Rho]} U''$. Then, using the Taylor expansion
  \[
    |U'(\Rho_{\bm i + \bm e_k}^{n+1}) - U'(\Rho_{\bm i}^{n+1})| \ge  \mu |\Rho_{\bm i + \bm e_k}^{n+1} - \Rho_{\bm i}^{n+1}|
  \]
  and hence we have that
  \begin{equation*}
    \left| \frac{\Rho_{\bm i + \bm e_k}^{n+1} - \Rho_{\bm i}^{n+1}}{h_k} \right|^2 \le \frac 1 {\mu^{2}} \left| \frac{U'(\Rho_{\bm i + \bm e_k}^{n+1}) - U'(\Rho_{\bm i}^{n+1})}{h_k} \right|^2
  \end{equation*}
  Since $\underline \Rho \le \rho_\seqi \le \overline \Rho$ on the admissible cells uniformly for all $\seqi$, then in \Cref{cor:U'(rho) is H1x} the right-hand side is bounded uniformly.

  \step[Bound in the $Y_\seqi$-norm]
  This is precisely \Cref{lem:time continuity}.

  \step[Convergence of $\rho$] We can apply \Cref{thm:gallouet,lem:gallouet in our setting} to show that, up to a subsequence, $\rho_\seqi$ converges to a function $\rho$ in $L^2( (0,T) \times \Omega)$. A further subsequence converges a.e.\ in $(0,T) \times \Omega$. Hence \ref{it:upper lower bound} holds.

  \step[Limit of $\xi$]
  Consider the function
  \begin{equation*}
    \xi_\seqi (t, x) =
    \begin{cases}
      \xi_{\bm i}^{n+1} & \text{if } n\tau_\seqi\le t < (n+1)\tau_\seqi\text{ and } x \in Q^{(\bm h_\seqi)} _{\bm i} \\
      0                 & \text{otherwise.}
    \end{cases}
  \end{equation*}
  Due to the a.e.\ convergence, the uniform upper and lower bounds, the assumption that $U\in C^2((0,\alpha))$, and using that $K$ is Lipschitz, together with continuity of time translations for the midpoint term, we recover that, up to a subsequence, $\xi_\seqi$ converges in $L^2((0,T) \times \Omega)$ to $\xi$ satisfying \eqref{eq:xi}.

  \step[Limit of the velocities $v$]
  We consider
  \begin{align*}
    v_{\seqi,k} (t, x) & =
    \begin{cases}
      v_{\bm i + \frac 1 2 \bm e_k}^{n+1}
      & \text{if } n\tau_\seqi\le t < (n+1)\tau_\seqi\text{ and } x \in Q^{(\bm h_\seqi)} _{\bm i}
      \\
      0
      & \text{otherwise.}
    \end{cases}
  \end{align*}
  We decompose, avoiding the index $\seqi$ for simplicity
  \begin{equation*}
    v_{\bm i + \frac 1 2 \bm e_k}^{n+1}
    = -\frac{U'(\Rho_{\bm i + \bm e_k}^{n+1}) - U'(\Rho_{\bm i}^{n+1})}{h_k}
    - \frac{V_{\bm i + \bm e_k} - V_{\bm i}}{h_k}
    - {|Q^{(\bm h)}|} \sum_{\bm j } \frac{K_{\bm i + \bm e_k, \bm j} - K_{\bm i, \bm j}}{h_k}  \Rho_{\bm j}^{n+\frac 1 2}.
  \end{equation*}
  \Cref{cor:U'(rho) is H1x} shows that the first term is bounded in $L^2$, and the second and third terms are bounded in $L^\infty$ since $\nabla V$ and $\nabla_x K$ are bounded.
  Therefore, up to a further subsequence $(v_{\seqi,k})_+$ and $(v_{\seqi,k})_-$ converge to limits, $v_{+,k}, v_{-,k}$, weakly in $L^2 ((0,T) \times \Omega)$.
  Using test functions it is easy to see that $v_{+,k} + v_{-,k} = -\frac{\partial \xi}{\partial x_k}$.
  \step[Limit of the mobility]
  Let us fix a direction $k \in \{1,\cdots, d\}$.
  Consider $\tilde \rho_\seqi$ extended by zero outside of $\Omega$.
  The subsequence we are studying converges strongly in $L^2((0,T) \times \Rd)$.
  The Fréchet--Kolmogorov theorem ensures that
  \begin{equation*}
    \omega_k(\sigma) \defeq \sup_{\seqi} \| \tilde \rho_\seqi(\cdot + \sigma \bm e_k) - \tilde \rho_\seqi \|_{L^2((0,T) \times \Rd)}
  \end{equation*}
  is a modulus of continuity.

  Let $\omega \Subset \Omega$. For $h_\seqi$ small enough the piece-wise constant interpolation of $\mobupwind(\bm P_{\bm i}^{n+1}, \bm P_{\bm i + \bm e_k}^{n+1} )$ given by
  \begin{align*}
    f_{\seqi,k} (t, x) & =
    \begin{cases}
      \mobupwind(\Rho_{\bm i}^{n+1}, \Rho_{\bm i + \bm e_k}^{n+1})
      & \text{if } n\tau_\seqi\le t < (n+1)\tau_\seqi\text{ and } x \in Q^{(\bm h_\seqi)} _{\bm i}, \bm i + \bm e_k \in \bm I_{\bm h_\seqi} (\Omega)
      \\
      0
      & \text{otherwise.}
    \end{cases}
  \end{align*}
  is precisely $\mobupwind(\rho_\seqi, \rho_\seqi(\cdot + h_k \bm e_k))$.
  Due to the modulus of continuity deduced above, and the continuity of ${\mobupwind}$ we deduce that $f_{\seqi, k}$ converges strongly in $L^2((0,T) \times \omega)$
  to $\mob(\rho)$.
  Since this holds for any $\omega \Subset \Omega$ and the sequence is bounded in $L^\infty ((0,T) \times \Omega)$, the convergence holds in $L^2((0,T) \times \Omega)$.
  Similarly for
  \begin{equation*}
    g_{\seqi,k} (t, x) =
    \begin{cases}
      \mobupwind(\Rho_{\bm i + \bm e_k}^{n+1}, \Rho_{\bm i}^{n+1})
      & \text{if } n\tau_\seqi\le t < (n+1)\tau_\seqi\text{ and } x \in Q^{(\bm h_\seqi)} _{\bm i}, \bm i + \bm e_k \in \bm I_{\bm h_\seqi} (\Omega)
      \\
      0
      & \text{otherwise.}
    \end{cases}
  \end{equation*}

  \step[Limit of the fluxes $F$]
  We define
  \begin{align*}
    F_{\seqi,k} (t, x) & =
    \begin{cases}
      F_{\bm i + \frac 1 2 \bm e_k}^{n+1} & \text{if } n\tau_\seqi\le t < (n+1)\tau_\seqi\text{ and } x \in Q^{(\bm h_\seqi)} _{\bm i} \\
      0                                   & \text{otherwise.}
    \end{cases}
  \end{align*}
  Hence, due to the up-winding and the weak convergence of $(v_{\seqi,k})_{\pm}$ we have that
  \begin{equation*}
    F_{\seqi,k} = f_{\seqi, k} (v_{\seqi, k})_+ + g_{\seqi, k} (v_{\seqi, k})_- \rightharpoonup \mob (\rho) v_{+,k} + \mob(\rho) v_{-,k} = - \mob(\rho) \frac{\partial \xi}{\partial x_k}, \qquad \text{weakly in } L^1 ((0,T) \times \Omega).
  \end{equation*}

  \step[The limit is a weak solution]
  Let $\varphi \in C^\infty_c([0,T) \times \overline \Omega)$ be extended by zero for $t\ge T$, and set $N_\seqi=\lceil T/\tau_\seqi\rceil$, $t_n=n\tau_\seqi$.
  Let us define
  \begin{equation*}
    \varPhi_{\bm i}^{n} \defeq \varphi(t_{n},x_{\bm i}^{(\bm h_\seqi)}){,\qquad 0\le n\le N_\seqi.}
  \end{equation*}
  Multiply each equation in \eqref{eq:scheme} by $\varPhi_{\bm i}^n$ and add to recover
  \begin{align*}
    \tau_\seqi\sum_{n=0}^{N_\seqi-1} \sum_{\bm i \in \bm I_{\bm h_\seqi} (\Omega) } \frac{\Rho_{\bm i}^{n+1}- \Rho_{\bm i}^{n}}{\tau_\seqi} \varPhi_{\bm i}^{n}
    + \tau_\seqi\sum_{n=0}^{N_\seqi-1} \sum_{k = 1}^d \sum_{\bm i \in \bm I_{\bm h_\seqi} (\Omega) } \frac{F_{\bm i + \frac 1 2 \bm e_k}^{n+1} - F_{\bm i - \frac 1 2 \bm e_k}^{n+1}}{h_k} \varPhi_{\bm i}^{n} =0.
  \end{align*}
  For $\seqi$ large enough $\Phi_{\bm i}^{N_\seqi} = 0$ and
  exchanging the sums we recover
  \begin{align*}
    & - \tau_\seqi|Q^{(\bm h_\seqi)}|\sum_{n=0}^{N_\seqi-1} \sum_{\bm i \in \bm I_{\bm h_\seqi} (\Omega)} \Rho_{\bm i}^{n+1}\frac{\varPhi_{\bm i}^{n+1}- \varPhi_{\bm i}^{n}}{\tau_\seqi}
    \\
    & \qquad - \tau_\seqi|Q^{(\bm h_\seqi)}|\sum_{n=0}^{N_\seqi-1}  \sum_{k = 1}^d  \sum_{\substack{\bm i \in \bm I_{\bm h_\seqi} (\Omega) \text{ s.t.}                           \\\bm i + \bm e_k \in \bm I_{\bm h_\seqi} (\Omega)}}
    F_{\bm i + \frac 1 2 \bm e_k}^{n+1} \frac{ \varPhi_{\bm i + \bm e_k}^{n} - \varPhi_{\bm i}^{n}}{h_k}
    = |Q^{(\bm h_\seqi)}|\sum_{{\bm i\in}\bm I_{\bm h_\seqi} (\Omega)}{\Rho_{\bm i}^{0} } \varPhi_{\bm i}^{0} .
  \end{align*}
  For $\tau_\seqi$ small enough $\varphi(t_{N_\seqi}, \cdot) = 0$, therefore
  \begin{align*}
    |Q^{(\bm h_\seqi)}|\sum_{\bm i \in \bm I_{\bm h_\seqi} (\Omega)}\Rho_{\bm i}^{N_\seqi}\varPhi_{\bm i}^{N_\seqi} = 0.
  \end{align*}
  We have, recalling the notation \eqref{eq:definition of Omega_h}, that
  \begin{align*}
    \tau_\seqi|Q^{(\bm h_\seqi)}|\sum_{n=0}^{N_\seqi-1} \sum_{\bm i \in \bm I_{\bm h_\seqi} (\Omega)} \Rho_{\bm i}^{n+1}\frac{\varPhi_{\bm i}^{n+1}- \varPhi_{\bm i}^{n}}{\tau_\seqi}
    & = \sum_{n=0}^{N_\seqi-1} \sum_{\bm i \in \bm I_{\bm h_\seqi} (\Omega)} \Rho_{\bm i}^{n+1} \int_{t_n}^{t_{n+1}} \int_{Q^{(\bm h_\seqi)}_{\bm i} }\frac{\partial  \varphi}{\partial t} (t, x) dt dx + R_1
    \\
    & = \int_0^{T} \int_{\Omega_{\bm h_\seqi} } \rho_\seqi (t,x) \frac{\partial  \varphi}{\partial t} (t,x) dt dx
    + R_1 + R_2
    \\
    & = \int_0^{T} \int_{\Omega} \rho_\seqi (t,x) \frac{\partial \varphi}{\partial t} (t,x) dt dx
    + R_1 + R_2,
  \end{align*}
  where we can estimate the remainders
  \begin{align*}
    |R_1|
    & \le C \tau_\seqi N_\seqi |\Omega_{\bm h_\seqi}| \|\bm P\|_{\ell^\infty} \left( \left\| \nabla_x  \frac{\partial \varphi}{\partial t} \right\|_{L^\infty} \| \bm h_\seqi \|_{\infty} + \left\| \frac{\partial^2 \varphi}{\partial t^2} \right\|_{L^\infty} \tau_\seqi \right) \\
    |R_2| & \le C |T - \tau_\seqi N_\seqi| |\Omega_{\bm h_\seqi}| \|\bm P\|_{\ell^\infty} \left\| \frac{\partial \varphi}{\partial t} \right\|_{L^\infty}
  \end{align*}
  and $\Omega_{\bm h_\seqi} \subset \Omega$ and $\rho_\seqi = 0$ in $[0,T] \times (\Omega \setminus \Omega_{\bm h_\seqi})$.
  Similarly, we have that
  \begin{align*}
    & \tau_\seqi|Q^{(\bm h_\seqi)}|\sum_{n=0}^{N_\seqi-1}
    \sum_{k = 1}^d  \sum_{\substack{\bm i \in \bm I_{\bm h_\seqi} (\Omega) \text{ s.t.} \\\bm i + \bm e_k \in \bm I_{\bm h_\seqi} (\Omega)}}
    F_{\bm i + \frac 1 2 \bm e_k}^{n+1} \frac{ \varPhi_{\bm i + \bm e_k}^{n} - \varPhi_{\bm i}^{n}}{h_k}
    = \int_0^T \int_\Omega \sum_{k=1}^d F_{\seqi,k} \frac{\partial \varphi}{\partial x_k} dt dx + R_4,
  \end{align*}
  with $R_4=O(\|\bm h_\seqi\|_\infty+\tau_\seqi)$ by the uniform $L^1$ flux bound.
  By the convergences proven above we deduce that
  \begin{equation*}
    \int_0^T \int_\Omega \left(-\rho\frac{\partial \varphi}{\partial t} + \mob(\rho) \nabla \xi \cdot \nabla \varphi  \right) = \int_\Omega \rho_0(x) \varphi(0,x) dx
    \qquad \forall \varphi \in C_c^\infty([0,T) \times \overline \Omega).
    \qedhere
  \end{equation*}
  This concludes the proof. \qed
\end{enumeratesteps}

\subsection{Proof of \texorpdfstring{\Cref{thm:convergence with free boundary}}{Theorem \ref{thm:convergence with free boundary}}}

Since we only assume \ref{it:rho0 between 0 and alpha}, we can no longer use \cref{step:proof convergence data away from 0 - upper/lower bound}, and free boundaries may form.
Hence, we have to adapt the proof of \cref{step:proof convergence data away from 0 - bound Xj} by noting that
\begin{equation*}
  \left| \frac{\Rho_{\bm i + \bm e_k}^{n+1} - \Rho_{\bm i}^{n+1}}{h_k} \right|^2 \le {\frac 1 {\big(\inf_{s \in (0,\alpha)} U''(s)\big)^2}} \left| \frac{U'(\Rho_{\bm i + \bm e_k}^{n+1}) - U'(\Rho_{\bm i}^{n+1})}{h_k} \right|^2
\end{equation*}
To recover the uniform bound in \Cref{cor:U'(rho) is H1x} we have assumed that $\Lambda_U \in C([0,\alpha])$.
In order to pass to the limit in $\xi$ we use the assumption $U' \in C([0,\alpha])$.
The remaining steps apply directly to this setting, since they do not use any of the sharper hypothesis.
\qed

\section{Numerical examples}
\label{sec:numerical examples}
This article is mainly concerned with the proof of convergence of the scheme. For the performance of the scheme in the (sufficiently general) case $\mob(\rho) = \rho \mobdown(\rho)$ as well as a discussion of asymptotics as $n \to \infty$ we point the reader to \cite{BailoCarrilloHu2020,BailoCarrilloHu2023}.
We make a further comment on these asymptotics in \Cref{sec:extensions}, \Cref{extension:asymptotics}.
For a phenomenological study we point the reader to \cite{BurgerInzunzaMuletVillada2019}.
In this section we fix $\bm h = (h, \cdots, h)$, $\tau = \tau(h)$, and we will denote \eqref{eq:rho seqi} by $\rho_h$.
We write $\Rho_{h,\bm i}^n$ when displaying the mesh size explicitly.

\bigskip

We provide some numerical experiments.
The numerical schemes were developed in \texttt{julia} and are available at
\begin{quote}
  \url{https://github.com/dgomezcastro/FVADE.jl}
\end{quote}
The implementation works for saturation and non-saturation cases (i.e., when $\mobdown$ does not vanish).
The implicit equations are solved with Newton-Raphson's method where the derivative is obtained using \texttt{ForwardDiff.jl}. This library uses dual numbers and stably differentiates Lipschitz functions.
For example, the convention is that $\frac{d}{dx}|_{x=0} \max\{x,0\} =0$.
When $K = 0$ the condition $\tau = h^2$ is usually sufficient for Newton's method to converge in 10 iterations or fewer. When $K \ne 0$ we need smaller $\tau$ to make the fixed-point iterations to converge.

In the implementation we replace $I_{\bm h}$ by $\tilde I_{\bm h}$ made of $\bm i \in \mathbb Z^d$ with $x_{\bm i}^{(\bm h)} \in \Omega$. This does not affect the convergence results when $\partial \Omega$ is smooth, since we can extend our test functions as $C^2$ functions in a neighbourhood of $\Omega$.

\subsection{Steady states of the drift-diffusion equation in dimension \texorpdfstring{$d=2$}{d=2}}

In order to show the behaviour in two different domains $\Omega$, in $d = 2$, we provide two examples with diffusion and confinement in \Cref{fig:steady states}.
\begin{figure}[H]
  \centering
  \begin{subfigure}[t]{0.48\textwidth}
    \centering
    \includegraphics[width=0.9\textwidth]{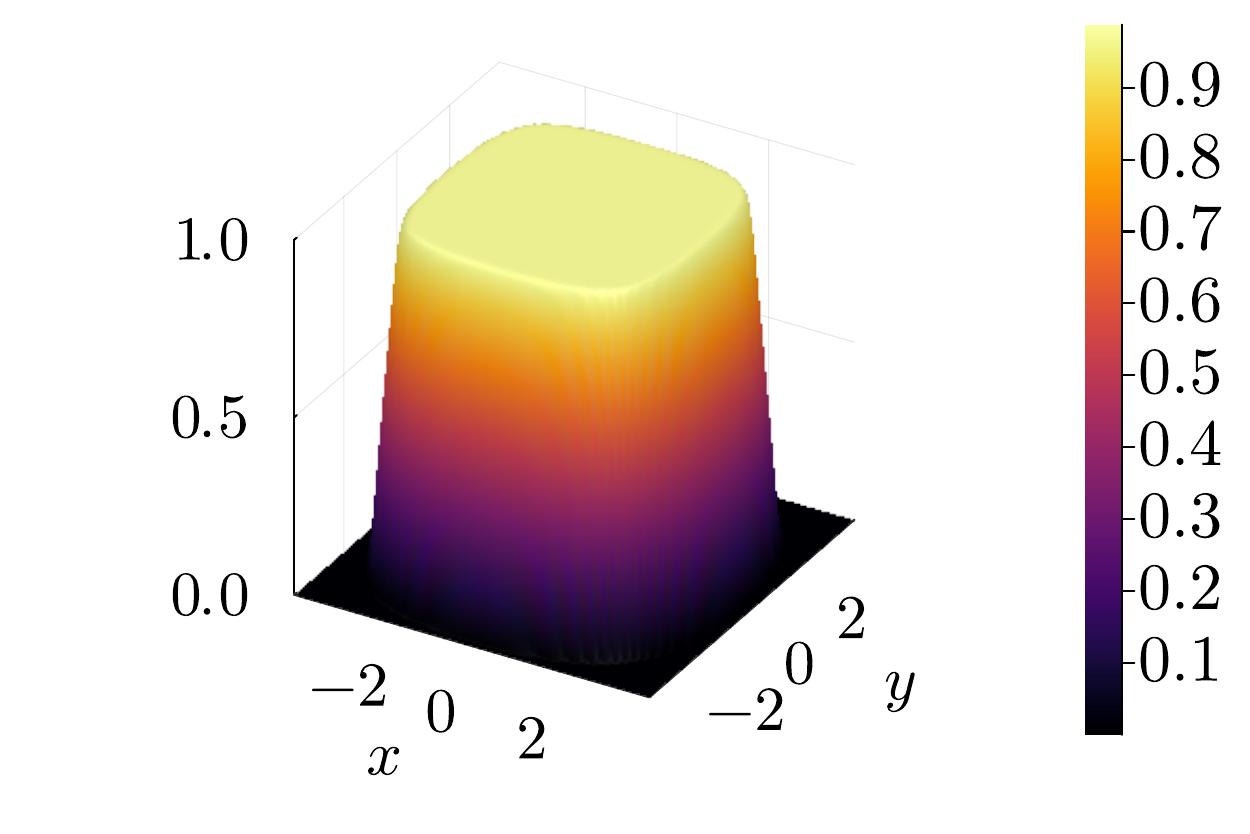}
    \caption{\input{driftdiffusion-saturation-square-caption.tex}}
    \label{fig:steady states square}
  \end{subfigure}
  \hfill
  \begin{subfigure}[t]{0.48\textwidth}
    \centering
    \includegraphics[width=0.9\textwidth]{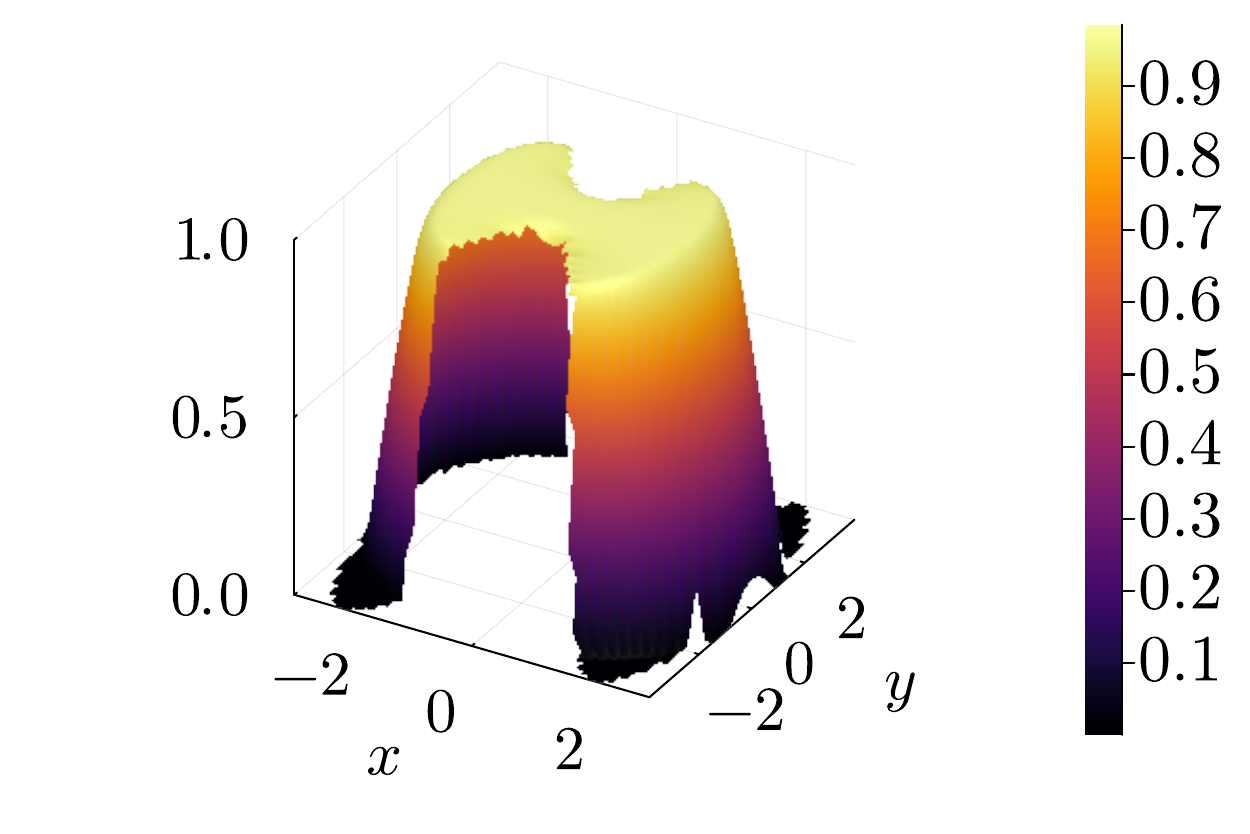}
    \caption{\input{driftdiffusion-saturation-peanut-caption.tex}}
    \label{fig:steady states peanut}
  \end{subfigure}
  \caption{Steady states in the domains $\Omega = [-4,4]^2$ and $\Omega = \{(x,y) : (x^2 - 3.9)^2 + y^2 < 4^2\}$ starting from $\rho_0(x) = 0.6$.}
  \label{fig:steady states}
\end{figure}

\subsection{Decay of the free energy with $K \ne 0$ in dimensions $d=1$ and $d=2$}

Next, we provide an example with aggregation to show the decay of the free energy in \Cref{fig:energy decay}. In this example we take $K \ne 0$.
Notice that the free energy stabilizes at the free energy of the attractor.
Furthermore, if the free energy stabilizes, then in \eqref{eq:decay of free energy} we observe that $\theta_{\bm i + \frac 1 2} |v_{\bm i + \frac 1 2}| \approx 0$. Hence, we are at a formal steady state.

\begin{figure}[H]
  \centering
  \begin{subfigure}[t]{0.48\textwidth}
    \centering
    \includegraphics[width=0.9\textwidth]{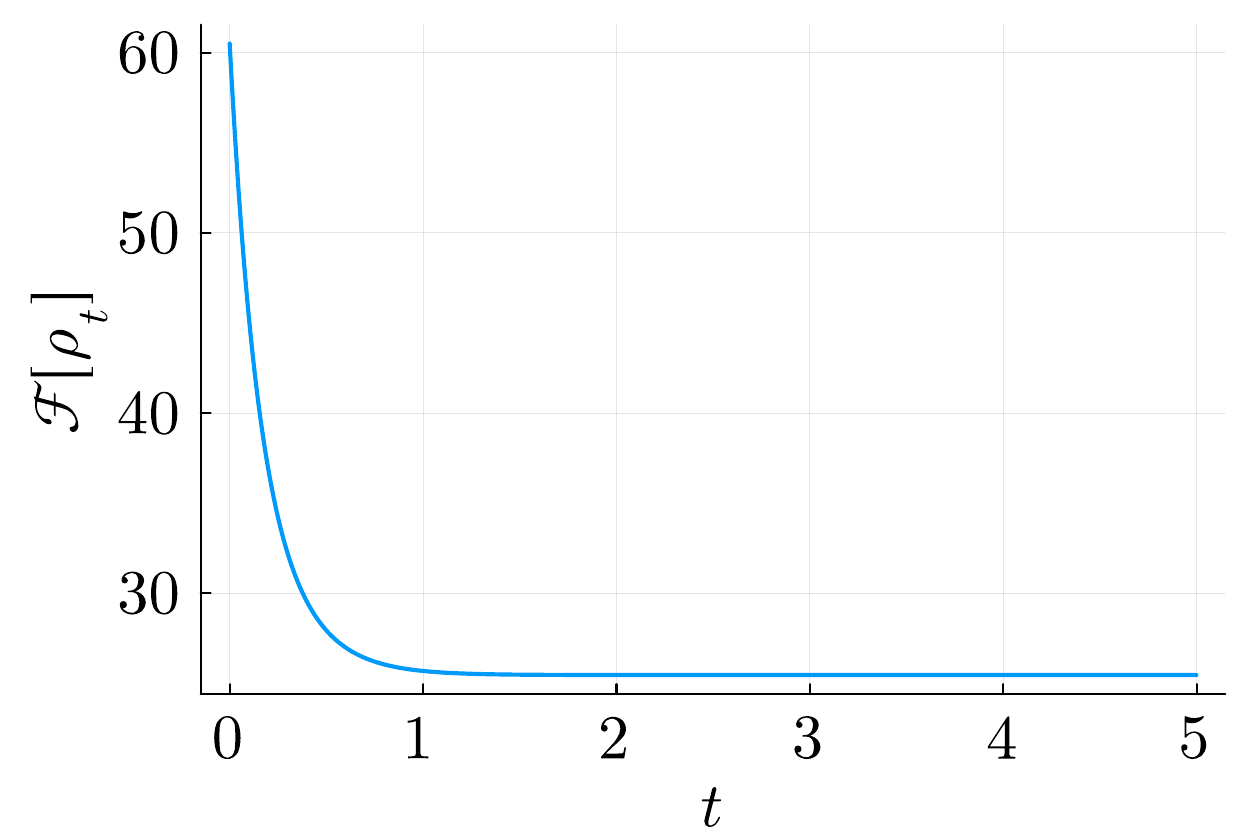}
    \caption{\input{energydecay-aggregation-d1-caption.tex}}
    \label{fig:energy decay d1}
  \end{subfigure}
  \hfill
  \begin{subfigure}[t]{0.48\textwidth}
    \centering
    \includegraphics[width=0.9\textwidth]{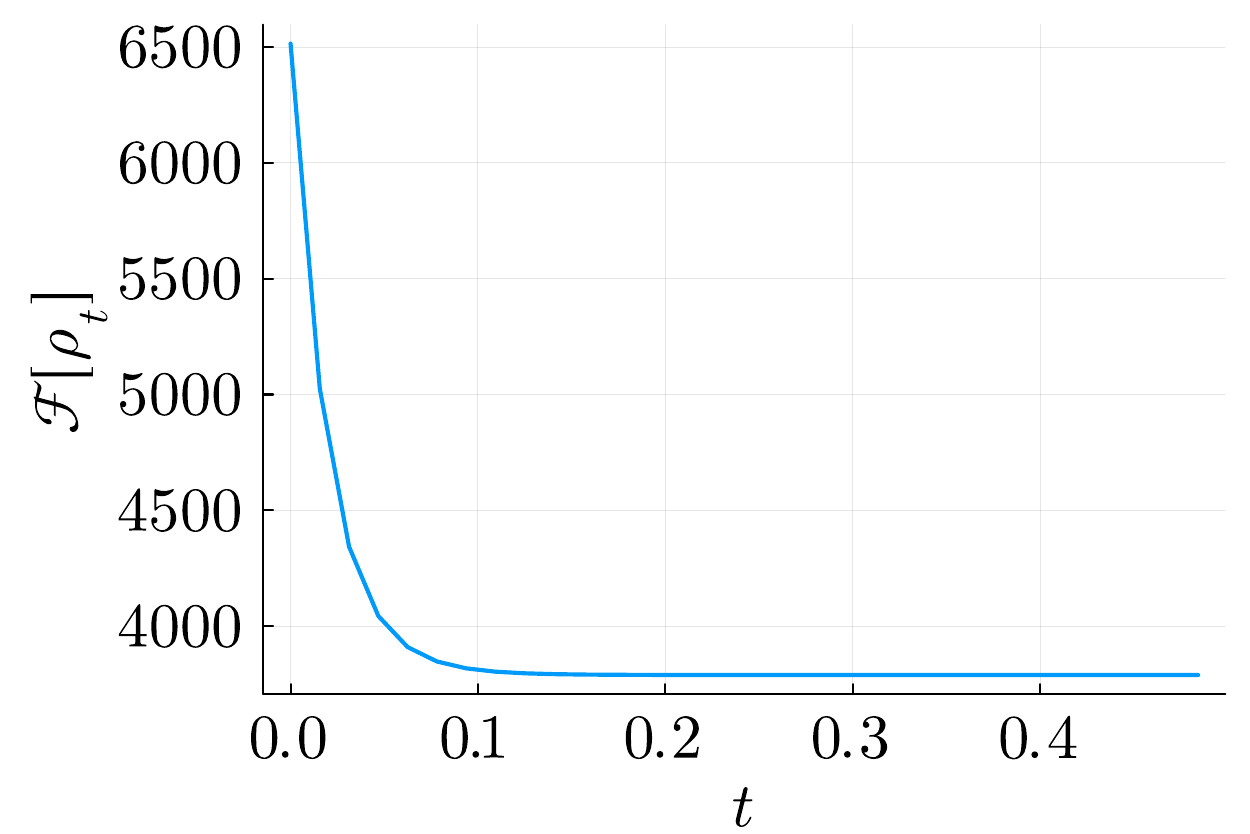}
    \caption{\input{energydecay-aggregation-d2-caption.tex}}
    \label{fig:energy decay d2}
  \end{subfigure}
  \caption{Decay of the free energy for $\Omega = [-4,4]$ and $\Omega = [-4,4]^2$.
    The large energy values are due to the fact that $\int_\Omega \int_\Omega |x-y|^2 dx dy$ is large.
  }
  \label{fig:energy decay}
\end{figure}

\subsection{Convergence as \texorpdfstring{$h \to 0$}{h→0} for the Porous-Medium Equation}
We will show the convergence to zero of the values
\begin{equation*}
  e_h \defeq \| \rho_h - \rho \|_{L^1((0,T) \times \Omega)}
\end{equation*}
The only example where we have an explicit solution of \eqref{eq:main problem} that is not stationary is the Porous-Medium Equation: linear mobility $\mob(\rho) = \rho$ and  $U(\rho) = \frac{\rho^m}{m-1}, V = K = 0$.
This is not a case with saturation, but we will nevertheless use it as a benchmark.
Notice that, even though $\rho$ and $\rho_h$ are explicit, it is not easy to compute $e_h$ exactly.
Since $\rho_h$ is piece-wise constant, and $\rho$ is Hölder continuous we will use the approximation
\begin{equation*}
  \varepsilon^{(1)}_h \defeq \tau_h h^d \sum_{n=0}^{N_h} \sum_{\bm i \in \bm I_h} |\Rho_{h,\bm i}^{n} - \rho(n\tau_h, h \bm i)|.
\end{equation*}
The convergence data is shown in \Cref{fig:numerics Barenblatt}.
\begin{figure}[H]
  \centering
  \begin{subfigure}[t]{0.48\textwidth}
    \centering
    \includegraphics[width=0.9\textwidth]{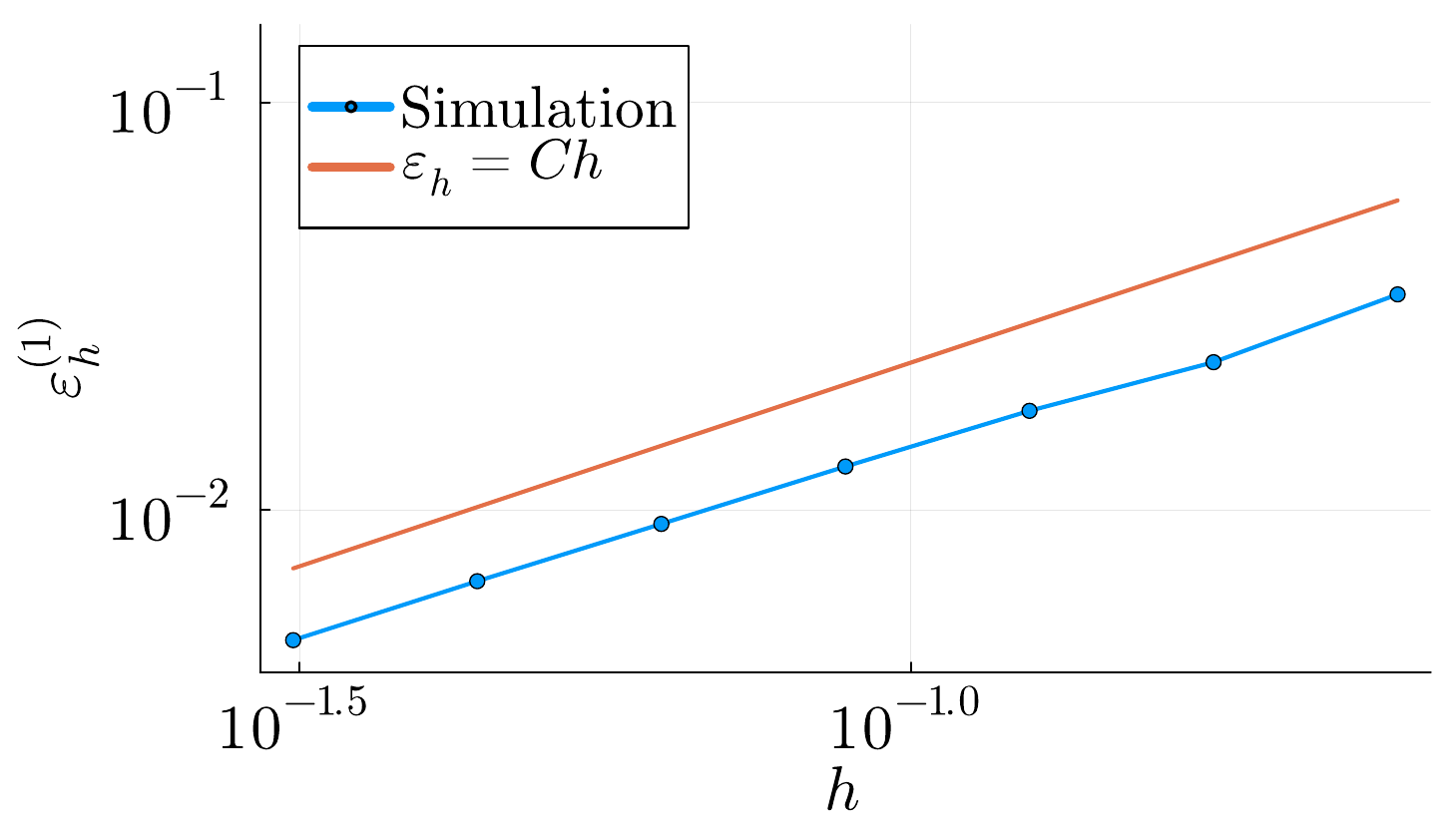}
    \caption{\input{convergence-Barenblatt-1d-caption.tex}}
    \label{fig:numerics Barenblatt 1d}
  \end{subfigure}
  \hfill
  \begin{subfigure}[t]{0.48\textwidth}
    \centering
    \includegraphics[width=0.9\textwidth]{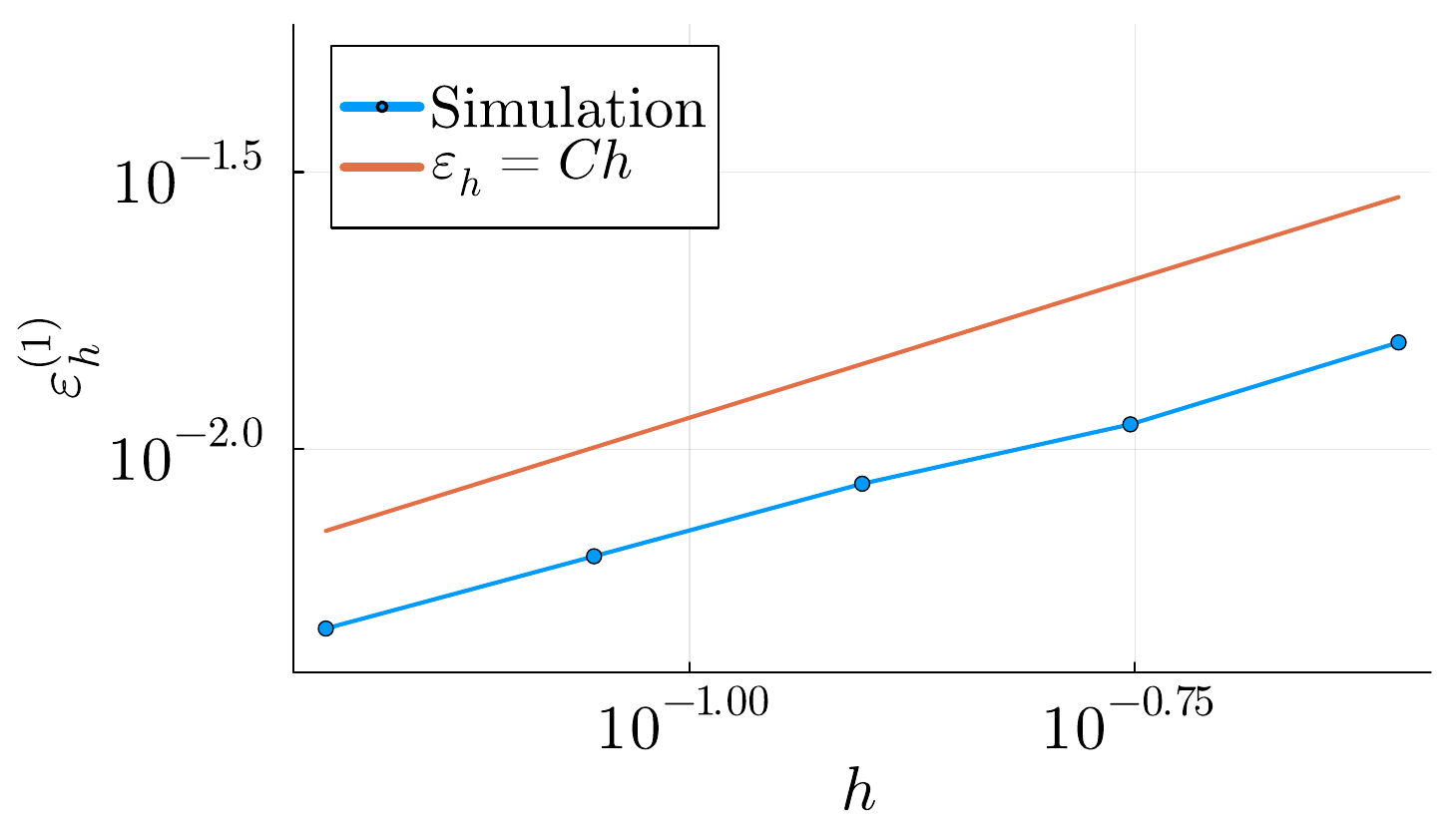}
    \caption{\input{convergence-Barenblatt-2d-caption.tex}}
    \label{fig:numerics Barenblatt 2d}
  \end{subfigure}
  \caption{Convergence rate as $h\to 0$ for Porous-Medium Equation ($\mob(\rho) = \rho$, $U(\rho)=\rho^2$, $V = K = 0$) in dimension $d=1$ and $d=2$ with $\rho_0(x) = B_M(t_0,x)$ the self-similar solution at time $t_0 = 1$ and mass $M = 2$.}
  \label{fig:numerics Barenblatt}
\end{figure}

\subsection{Convergence as \texorpdfstring{$h \to 0$}{h→0} in a drift-diffusion equation with saturation}
When there is no explicit solution to compare against, we can study the convergence rate by splitting $h$ in half. Informally speaking, assume that the numerical error $e_h \approx Ch^{\beta}$ with $\beta > 0$. Then $e_{h/2} \approx 2^{-\beta} e_h$.
Let us consider the error of splitting $h$ in half
\[
  e_h^{(2)} \defeq \| \rho_h - \rho_{h/2} \|_{L^1((0,T) \times \Omega)}.
\]
Using the triangle inequality we have that
\begin{equation*}
  e_h^{(2)} \le e_h + e_{h/2} \lesssim (1 + 2^{-\beta}) e_h, \qquad
  e_h \le e_h^{(2)} + e_{h/2} \lesssim e_h^{(2)} + 2^{-\beta} e_h.
\end{equation*}
We conclude that
\begin{equation}
  (1-2^{-\beta}) e_h \lesssim e_h^{(2)} \lesssim (1 + 2^{-\beta}) e_h.
\end{equation}
Thus, in informal terms, $e_h^{(2)}$ is a good estimator for the order of convergence and, furthermore, it is computable.
Since $2\bm I_h \subset \bm I_{h/2}$ and $\tau_h = 2^p\tau_{h/2}$ for $\tau_h=h^p$, $p\in\mathbb N$, we compare the discrete solutions at matching nodes and time levels:
\begin{equation*}
  \varepsilon_h^{(2)}
  \defeq \tau_h h^d \sum_{n=0}^{N_h} \sum_{\bm i \in \bm I_h} |\Rho_{h,\bm i}^{n} - \Rho_{h/2,2\bm i}^{2^p n}|.
\end{equation*}
We provide a numerical experiment testing the convergence rate in $d=1$ in an example with asymptotic limit exhibiting free boundary at levels $\rho = 0$ and $\rho = 1$, see \Cref{fig:convergence rates}. The convergence rate seems to be linear.
\begin{figure}[H]
  \centering
  \includegraphics[width=0.45\textwidth]{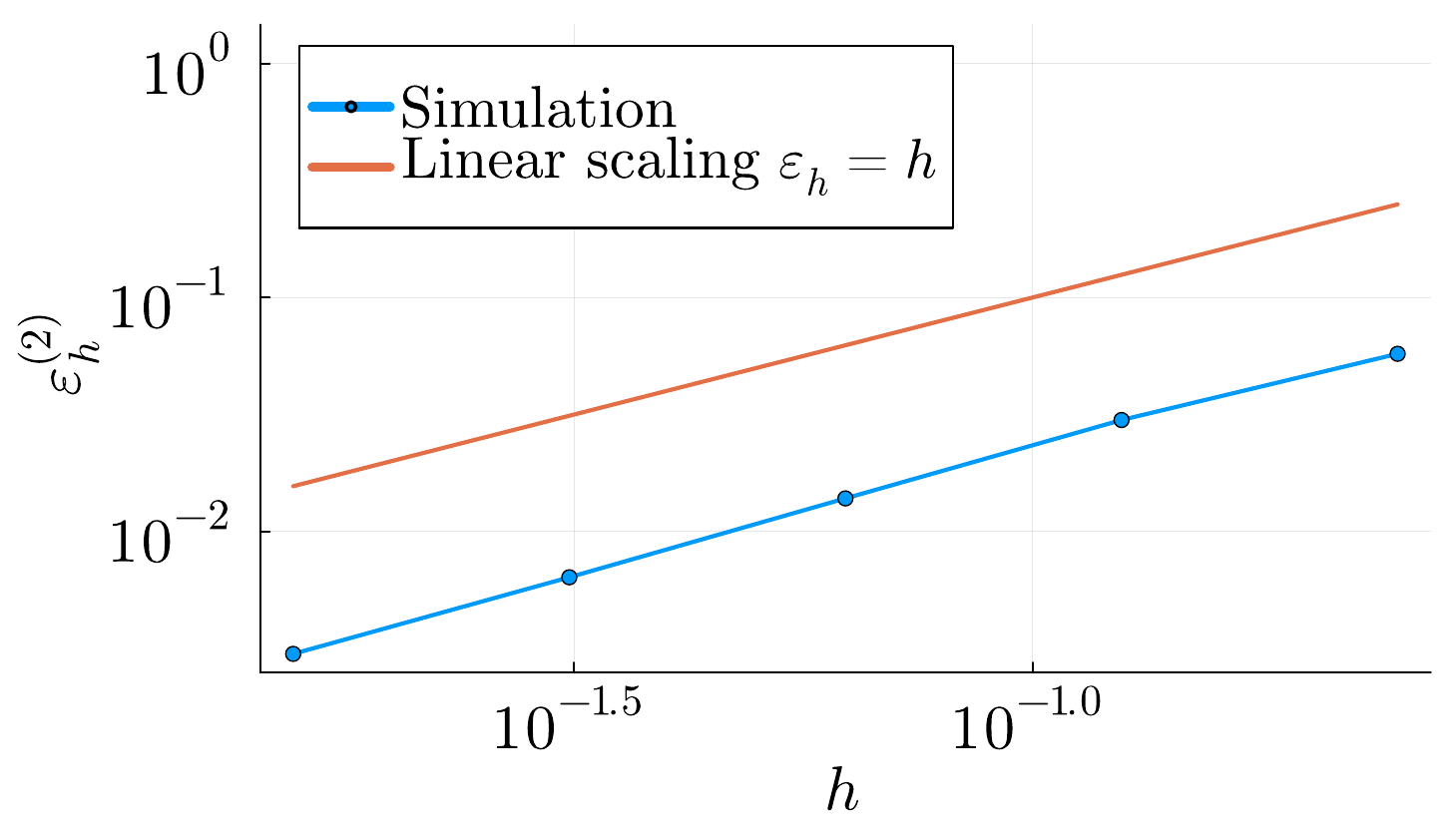}
  \caption[Convergence rate as $h \to 0$ for a drift-diffusion equation with saturation in dimension $d=1$]{Convergence rate as $h \to 0$ for a drift-diffusion equation with saturation in dimension $d=1$, using as initial datum a constant: \input{convergence-hhalf-d1-caption.tex}}
  \label{fig:convergence rates}
\end{figure}
The computation cost increases significantly when $K \ne 0$, and we do not present convergence tests in this case.

\section{Extensions and open problems}
\label{sec:extensions}
We provide several comments and open questions
\begin{enumerate}
  \item In \Cref{thm:convergence for data away from 0 mob Lipschitz,thm:convergence with free boundary} we provide examples of assumptions of $U, V, K$ so that convergence holds. We have focused on Porous-Medium-type diffusion problems (where $U \ne 0$), due to our $H^1$-space compactness approach. More general settings are interesting.

    For example, the Keller--Segel model or the Porous--Medium Equation with non-local pressure (see, e.g., \cite{stanFiniteInfiniteSpeed2016,stan2019ExistenceWeakSolutions}) use singular kernels which are not Lipschitz and, in fact, may be singular at the origin. An interesting open problem is to identify more general hypotheses under which the convergence results presented here still hold.
    An extension of the results of this paper to the Porous--Medium Equation with non-local pressure is a work in preparation by the author.

  \item
    In this paper we prove convergence by compactness methods. Our numerical experiments suggest that the convergence is linear in $h$, but this rate is not rigorously known. Proving this result is an interesting, and difficult, open problem.

  \item
    This paper focuses on rectangular Cartesian grids $\prod_{k=1}^d h_k\mathbb Z$. These grids are not well suited to some domains. For example, see \Cref{fig:steady states}. It is well known that finite-volume methods can be developed on general grids (see, e.g., \cite{EymardGallouetHerbin1997}).
    An open problem is whether the scheme of this paper can be adapted to this setting while preserving the suitable a priori estimates.

  \item \label{extension:asymptotics}
    In \cite{CarrilloFernandez-JimenezGomez-Castro}, the authors discuss asymptotics in discrete time (i.e., $n \to \infty$) when $K = 0$, and show that the numerical scheme has a time limit.

    If $\Rho^n\to\Rho^\infty$ and the chemical potentials have finite limits, free-energy dissipation implies that $\Rho^\infty$ is a steady state of the scheme.
    Hence
    $\theta_{\bm{i} + \tfrac 1 2 \bm{e}_k} |v_{\bm{i} + \tfrac 1 2 \bm{e}_k}^\infty| = 0$ for all $\bm{i}$.
    Therefore, when $\Rho_{\bm{i}}^\infty, \Rho_{\bm{i} + \bm{e}_k}^\infty \in (0,\alpha)$, then $\xi_{\bm{i}}^\infty = \xi_{\bm{i} + \bm{e}_k}^\infty$.
    Hence, on each connected component of cells with densities in $(0,\alpha)$, the discrete Euler--Lagrange condition holds with the interaction quadrature. When free boundaries are present, the analysis is more complicated (see \cite{CarrilloFernandez-JimenezGomez-Castro}).

    A detailed study of this
    phenomenon and the characterization of the numerical steady states in
    cases with free boundaries are interesting open problems.

  \item This scheme can be extended to systems of equations.

\end{enumerate}

\section*{Acknowledgements}
The author is thankful to Félix del Teso and Espen R. Jakobsen for fruitful discussions during the preparation of this manuscript.
The research of DGC was supported by grants RYC2022-037317-I and PID2023-151120NA-I00 from the Spanish Government MCIN/AEI/10.13039/501100011033/FEDER, UE.

\appendix

\section{Existence. Proof of \texorpdfstring{\Cref{prop:existence}}{Proposition~\ref{prop:existence}}}
\label{sec:existence proof}
For this proof we adapt the argument in \cite{CarrilloFernandez-JimenezGomez-Castro} to the case $K \neq 0$.
We first give the proof under the assumptions of
\Cref{thm:convergence for data away from 0 mob Lipschitz}.

Let $\Rho^n$ be known, and let us construct $\Rho^{n+1}$.
With $\lambda$ and $L$ as in \Cref{lem:free boundaries 2}, choose
\[
  0<\varepsilon<\frac{\min\{\min_{\bm i}\Rho_{\bm i}^n,\,\alpha-\max_{\bm i}\Rho_{\bm i}^n\}}{1+C_R\tau\lambda dL}.
\]
Let us define
\[
  D := {(\varepsilon/2,\alpha-\varepsilon/2)^{\bm I}.}
\]
Clearly $D$ is an open set, and $\bm\Rho^n\in D$.
We will use the notation in \cite[Section 7]{CarrilloFernandez-JimenezGomez-Castro}.

Let us now define the function
\[
  H_{\bm i}({\eta}, \bm \Rho) = \Rho_{\bm i} + {\eta} \tau \sum_{k=1}^d \frac{F_{\bm i + \frac 1 2 \bm e_k}(\bm \Rho) - F_{\bm i - \frac 1 2 \bm e_k}(\bm \Rho)}{h_k}
\]
and the functions $F_{\bm i + \frac{1}{2} \bm e_k}$ are defined naturally by the scheme \eqref{eq:scheme}.
We consider the function $H = (H_{\bm i})_{\bm i \in \bm I} : [0,1] \times \overline{D} \to \mathbb R^{\bm I}$, which is continuous.

Assume, towards a contradiction, that $\bm \Rho^n {=} H({\eta}, \bm y)$ with $\bm y \in \partial D$ and ${\eta} \in [0,1]$.
Since $H(0,\cdot)=\mathrm{id}$, $\eta>0$.
In that case, $\bm y$ is a solution to \eqref{eq:scheme} with time step $\widetilde{\tau} = {\eta} \tau$.
Using mass conservation
\begin{equation*}
  \sum_{\bm i} y_{\bm i} = \sum_{\bm i} \Rho_{\bm i}^{n}.
\end{equation*}
Thus the drift bound $\lambda$ is uniform in $\eta$, and \Cref{lem:free boundaries 2} gives $y_{\bm i}\in(\varepsilon,\alpha-\varepsilon)$ since $\widetilde\tau\le\tau$.
We conclude that $\bm y \in D$. Therefore, the topological degree is constant in ${\eta}$ and hence
\[
  \operatorname{deg}(H(1, \cdot), D, \bm \Rho^n) = \operatorname{deg}(H(0, \cdot), D, \bm \Rho^n) =
  \operatorname{deg}(\mathrm{id}, D, \bm \Rho^n) = 1.
\]
Since the topological degree is not zero, there exists a solution $\bm \Rho^{n+1} \in D$ to $H(1, \bm \Rho^{n+1}) = \bm \Rho^n$.

The proof under the assumptions of \Cref{thm:convergence with free boundary} is similar, but in two steps: first construct $\varepsilon$-approximations of $\Rho^n$, $\mob, V$, and $K$ to be under the assumptions of \Cref{thm:convergence for data away from 0 mob Lipschitz}, and then pass to the limit $\varepsilon \to 0$. This last step is easy since non-negativity and finite mass implies compactness in $\R^{|\bm I|}$.
\qed

\emergencystretch=1em
\printbibliography

\end{document}